\documentclass[a4paper,twoside]{article}      % Comments after  % are ignored
\usepackage{amsmath,amssymb,amsfonts,amsthm,amscd,dsfont} % Typical maths resource packages
\usepackage{color}
\usepackage{enumerate}
\usepackage{graphicx}
\usepackage{fullpage}
\usepackage{mathrsfs}
\usepackage{colonequals}
\usepackage{authblk} 

\newtheorem{theorem}{Theorem}[section]
\newtheorem{proposition}[theorem]{Proposition}
\newtheorem{corollary}[theorem]{Corollary}
\newtheorem{lemma}[theorem]{Lemma}
\theoremstyle{definition}
\newtheorem{remark}[theorem]{Remark}

\newtheorem{assumption}[theorem]{Assumption}

\def\div{\mathop{\mathrm{div}}\nolimits}

\def\R{\mathbb R}
\def\vep{\varepsilon}

\def\A{\mathcal A}
\def\S{\mathcal S}
\def\H{\mathbb H}
\def\Hc{\mathcal{H}}

\def\J{\mathbb J}
\def\x{\mathbf x}
\def\X{\mathbf X}
\def\K{\mathsf{K}}
\def\Q{\mathsf{Q}}
\def\Tr{\mathrm{Tr}}

\def\B{\mathcal{B}}
\def\D{\mathcal{D}}
\def\Y{\mathcal{Y}}
\def\Qc{\mathcal{Q}}

\def\Qu{\hat{Q}}
\def\Pu{\hat{P}}
\def\Zu{\hat{Z}}
\def\Wu{\hat{W}}
\def\L{\mathscr{L}}
\def\qu{\hat{q}}
\def\pu{\hat{p}}
\def\zu{\hat{z}}

\def\aL{\mathsf L}
\def\aM{\mathsf M}
\def\aE{\mathsf E}
\def\aS{\mathsf S}
\def\aZ{\mathsf Z}
\def\az{\mathsf z}
\def\aF{\mathsf F}
\def\aG{\mathsf G}

\title{Mean field limits for non-Markovian interacting particles: convergence to equilibrium, GENERIC formalism, asymptotic limits and phase transitions
}
\author{M. H. Duong\thanks{m.duong@imperial.ac.uk}~ and G. A. Pavliotis\thanks{g.pavliotis@imperial.ac.uk}}
\affil{Department of Mathematics, Imperial College London, London SW7 2AZ, UK.}
\begin{document}
\maketitle
% Emails: {m.duong, g.pavliotis}@imperial.ac.uk.

\begin{abstract}

In this paper, we study the mean field limit of weakly interacting particles with memory that are governed by a system of non-Markovian Langevin equations. Under the assumption of quasi-Markovianity (i.e. that the memory in the system can be described using a {\bf finite} number of auxiliary processes), we pass to the mean field limit to obtain the corresponding McKean-Vlasov equation in an extended phase space. For the case of a quadratic confining potential and a quadratic (Curie-Weiss) interaction, we obtain the fundamental solution (Green's function). For nonconvex confining potentials, we characterize the stationary state(s) of the McKean-Vlasov equation, and we show that the bifurcation diagram of the stationary problem is independent of the memory in the system. In addition, we show that the McKean-Vlasov equation for the non-Markovian dynamics can be written in the GENERIC formalism and we study convergence to equilibrium and the Markovian asymptotic limit.

\end{abstract}

\section{Introduction}

Many systems in nature and in applications can be modeled using systems of interacting particles--or agents--that are possibly subject to thermal noise. Standard examples include plasma physics and galactic 
dynamics~\cite{BinneyTremaine2008}; more recent applications are dynamical density functional theory~\cite{GodPavlal-2012, GodPavlKall11}, mathematical biology~\cite{Farkhooi2017,Lucon2016} and  even mathematical models in the social sciences~\cite{GPY2017,Motsch2014}. As examples of models of interacting agents in a noisy environment that appear in the social sciences we mention the modelling of cooperative behavior~\cite{Dawson1983}, risk management~\cite{GPY2013} and  opinion formation~\cite{GPY2017}. 

For weakly interacting diffusions, i.e. when the strength of the interaction is inversely proportional to the number of particles, the mean field limit for the empirical measure has been studied rigorously, under quite general assumptions, see e.g.~\cite{DawsonGartner1987, gartner1988, oelschlager1984}. The mean field dynamics is described by the so-called McKean-Vlasov equation, which is a nonlinear, nonlocal Fokker-Planck type equation~\cite{McKean1966, McKean1967, BKRS2015}. This equation has been studied extensively and many things are known, including well-posedness of the evolution PDE~\cite{Chazelle2017}, convergence to equilibrium, (non)uniqueness of the invariant measure and phase transitions. For example, it is by now well known that, for nonconvex confining potentials and a quadratic (Curie-Weiss type) interaction term--the so-called Desai-Zwanzig model~\cite{DesaiZwanzig1978}, there exist more than one invariant measures at low temperatures \cite{Dawson1983, Shiino1987, Tugaut2014}. Similar results of nonuniqueness of  the stationary state at low temperatures have been also obtained for McKean-Vlasov equations modeling opinion formation~\cite{Chazelle_al2017, ChayesPanferov2010} as well as for the Desai-Zwanzig model in a two-scale potential~\cite{GomesPavliotis2018}.

Most works on the study of the McKean-Vlasov equation are concerned with the equation that is obtained in the mean field limit of weakly interacting overdamped or underdamped Langevin dynamics. The McKean-Vlasov equation corresponding to weakly interacting overdamped Langevin dynamics is a (nonlinear and nonlocal) uniformly elliptic parabolic PDE and many of its properties are by now well understood, see for instance~\cite{carrillo2003,Carrillo2006} and references therein. For weakly interacting underdamped Langevin diffusions, the corresponding McKean-Vlasov equation is not uniformly elliptic but, rather, hypoelliptic. However, many results are also known in this case, see for instance~\cite{DV2001, Villani2009} and references therein. A recent result that is particularly relevant for the present work is that the invariant measure(s) of the McKean-Vlasov equation in phase space have a product measure structure and, in particular that the overdamped and underdamped McKean-Vlasov dynamics exhibit the same phase transitions~\cite{DuongTugaut2016}. One of the goals of this paper is to extend this result to McKean-Vlasov equations corresponding to non-Markovian Langevin dynamics.

To our knowledge, the mean field limit of interacting particles that are driven by colored noise, or that have memory is much less studied. Both problems are quite interesting from a modeling perspective, given that noise in many physical systems exhibits a nontrivial (spatiotemporal) correlation structure, and that many interesting dynamical systems are non-Markovian. As examples where colored noise and non-Markovianity play an important role, we mention colloidal systems and polymer dynamics~\cite{snook07} and nonequilibrium systems and active matter~\cite{Fodor2016,Mandal2017}. The presence of non-white noise and of memory renders the analysis more complicated, since, even for a finite number of interacting agents, the corresponding Fokker-Planck equation does not have a gradient structure (for colored noise) and it is necessarily degenerate (for systems with memory~\cite{OttobrePavliotis11}). The main goal of this paper is to study the mean field limit and the properties of the resulting McKean-Vlasov equation for non-Markovian weakly interacting particle systems, under the assumption that the memory (or non-trivial correlation structure of the noise) can be described by means of a {\it finite} number of auxiliary processes--the so-called quasi-Markovian assumption~\cite{Pavl2014}[Ch. 8]. The study of mean field limits for weakly interacting particles driven by colored noise will be presented elsewhere.

Finite dimensional, and in particular, low dimensional stochastic systems with memory and or/colored noise have been studied extensively in the literature. We mention the work on noise induced transitions for Langevin equations driven by colored noise~\cite{HorsLef84, HanggiJung1995} and the study of the generalized Langevin equation, including the rigorous derivation of such non-Markovian equations for particles coupled to one or more heat baths~\cite{Rey-Bellet2006} and the references therein,~\cite{Pavl2014}[Ch. 8], the study of the ergodic properties of such systems~\cite{tropper1977}, of their long time behaviour~\cite{EPR99}, etc.

The starting point of the present work is the following particle system
\begin{subequations}\label{e:particles}
\begin{eqnarray}
dQ_i(t)&=&  P_i(t)\,dt
\\dP_i(t)&=&-\nabla V(Q_i(t))\,dt-\frac{1}{N}\sum_{j=1}^N\nabla_qU(Q_i(t)-Q_j(t))\,dt+\lambda^T Z_i(t)\,dt
\\ dZ_i(t)&=&-\lambda P_i(t) \,dt-AZ_i(t)+\sqrt{2\beta^{-1}A}dW_i(t).
\end{eqnarray}
\end{subequations}
Here, $i=1,\dots, N$, $(Q_i,P_i,Z_i)\in\R^d\times\R^d\times\R^{dm}=:\mathbf{X}$; $\lambda\in\R^{dm\times d}$ and $A\in\R^{dm\times dm}$ are given matrices with $A$ being symmetric positive definite; $\beta>0$ is the inverse temperature; and $V, U:\R^d\rightarrow\R$ are confining and interaction potentials, respectively.

This diffusion process in the extended phase space $\mathbf{X}$ is equivalent to the system of weakly interacting generalized Langevin equations, 
\begin{equation}\label{e:gle-inter}
\ddot{q}_i = -\frac{\partial V}{\partial q_i} - \frac{1}{N} \sum_{j=1}^N U'(q_i - q_j) + \sum_{j=1}^N \gamma_{i j} (t-s) \, \dot{q}_j(s) \, ds + F_i(t), \quad i =1,\dots N,
\end{equation}
where $F(t) = (F_1(t), \dots F_N(t))$ is a mean zero, Gaussian, stationary process with autocorrelation function
$$
E (F_i(t)  \, F_j(s)) = \beta^{-1} \gamma_{i j}(t-s).
$$
We can obtain the Markovian dynamics~\eqref{e:particles} from the generalized Langevin dynamics~\eqref{e:gle-inter}, under the assumption that the autocorrelation functions $[\gamma_{i j}(t-s)]_{i,j=1,\ldots,m}$ are given by a linear combination of exponential functions~\cite{OttobrePavliotis11}. In this case, the memory of the system can be described by adding a finite number of auxiliary statinary Ornstein-Uhlenbeck processes. Approximating arbitrary Gaussian stationary processes by a Markovian process in a rigorous, systematic and algorithmic way is still an open problem, related to the so-called stochastic realization problem~\cite{DymMcKean1976, LindquistPicci1985}, see also~\cite{Kup03}. A recent approach to this problem, based on the representation of the noise by an infinite dimensional Ornstein-Uhlenbeck process is presented in~\cite{Nguyen2018}, see also~\cite{OttobrePhD}. We plan to return to the combined Markovian approximation--mean field limit in future work.

It is possible to study the hydrodynamic (mean field) limit, i.e. to show that the empirical measure 
$$
\rho_N:=\frac{1}{N}\sum_{i=1}^N \delta_{(Q_i(t),P_i(t),Z_i(t))}
$$
converges to the solution of 
\begin{equation}
\label{eq: GLMV}
\partial_t \rho=-\div_q(p\rho)+\div_p\Big[(\nabla_q V(q)+\nabla_q U\ast\rho(q)-\lambda^T z)\rho\Big]+\div_z\Big[(p \lambda +Az)\rho\Big]+\beta^{-1}\div_z(A\nabla_z\rho).
\end{equation}
We will refer to~Equation~\eqref{eq: GLMV} as the \textit{generalized  McKean-Vlasov equation (GLMV)} for $\rho=\rho(t,q,p,z)$. It is  the forward Kolmogorov (Fokker-PLanck) equation associated to the following generalized McKean-Vlasov SDE
\begin{subequations}\label{e:mckean}
\begin{eqnarray}
dQ(t)&=&P(t)\,dt,
\\dP(t)&=&-\nabla V(Q(t))\,dt-\nabla_q U\ast\rho_t(Q_t)\,dt+\lambda^T Z(t)\,dt,
\\ dZ(t)&=&-P(t) \lambda\,dt-AZ(t)+\sqrt{2\beta^{-1}A}dW(t).
\end{eqnarray}
\end{subequations}
where $\rho_t=\mathrm{Law}(Q(t),P(t),Z(t))$. The rigorous passage from the particle approximation~\eqref{e:particles} to the generalized McKean-Vlasov SDE~\eqref{e:mckean} and the mean field limit for small Lipschitz interactions, can be justified, in principle,  using the coupling approach developed in~\cite{BolleyGuillinMalrieu2010}, see also~\cite{Duong2015NA}. The smallness restriction on the interaction can be removed using the results presented in~\cite{Monmarche2017}. These papers consider the case of interacting underdamped Langevin dynamics, but it should be possible to use similar techniques in order to study rigorously the mean field limit for the generalized Langevin interacting dynamics. We mention that the system interacting particles that we study in this paper is similar to systems of interacting nonlinear oscillators, coupled to two heat baths at different temperatures~\cite{EPR99}. Two related equations that also play an important role in this work are the \textit{overdamped McKean-Vlasov equation},
\[
\partial_t \rho=\div\Big[(\nabla V+\nabla U\ast\rho)\rho\Big]+\beta^{-1}\Delta\rho,
\]
and the \textit{underdamped McKean-Vlasov equation}
\[
\partial_t \rho=-\div_q(p\rho)+\div_p\Big[(\nabla_q V(q)+\nabla_q U\ast\rho(q))\rho\Big]+\gamma\div_p\Big[p \rho\Big]+\beta^{-1}\gamma\Delta_p\rho.
\]
The generalized Langevin equation, i.e. Equation~\eqref{e:mckean} in the absence of an interaction potential ($U\equiv 0$), has been studied extensively. In particular, the ergodic properties and the hypoelliptic and hypocoercive structure of the dynamics and of the corresponding Fokker-Planck equation have been analyzed in~\cite{OttobrePavliotis11} where the homogenization and overdamped limits are also studied. Another
recent paper that has motivated this work is \cite{Duong2017} where the overdamped limit of the underdamped McKean-Vlasov equation has been derived using variational and large deviations approach. In this paper, we generalize these aforementioned results to Equation~\eqref{e:mckean} for a quite broad class of confining and interaction potentials.  We obtain the fundamental solution (Green's function) for the generalized McKean-Vlasov equation~\eqref{eq: GLMV} for the case of a quadratic confining potential and a quadratic (Curie-Weiss) interaction. For nonconvex confining potentials, we characterize its stationary state(s) showing that the bifurcation diagram of the stationary problem is independent of the memory in the system. In addition, we study convergence to equilibrium and the Markovian asymptotic limit for both the finite system \eqref{e:particles} and Equation~\eqref{eq: GLMV} using the formal perturbation expansions method developed in \cite{PavliotisStuart2008,Pavl2014}. Furthermore, we demonstrate that Equation~\eqref{eq: GLMV} shares a similar GENERIC (\underline{G}eneral \underline{E}quation for \underline{N}on\underline{E}quilibrium \underline{R}eversible-\underline{I}rreversible \underline{C}oupling~\cite{Oettinger05}) structure, which is a well-established formalism for non-equilibrium thermodynamical systems unifying both reversible and irreversible dynamics, with the overdamped and underdamped McKean-Vlasov equations.

%More precisely, we are interested in understanding the effect of memory on the structure of the bifurcation diagram for the stationary dynamics, on the convergence to equilibrium, the emergence of non-equilibrium steady states and the white noise limits. 

%The generalized McKean-Vlasov equation~\eqref{eq: GLMV}, shares some of the properties of the McKean-Vlasov equation corresponding to the underdamped dynamics  
%
%{\bf insert equation}
%
%In particular, we show in this paper that {\bf product measure structure, GENERIC, overdamped limits}

The rest of the paper is organized as follows. In Section~\ref{sec:quadratic}, we calculate the spectrum of the Fokker-Planck operator associated to the finite system~\eqref{e:particles} and construct the fundamental solution to the mean field equation~\eqref{eq: GLMV} in the case of quadratic confining and interaction potentials. In Section~\ref{sec: nonconvex} we study nonconvex confining potentials. We provide a characterization of stationary states, study phase transition and show convergence to a stationary state. In Section~\ref{sec:GENERIC} we recast the generalized McKean-Vlasov equation into the GENERIC framework, obtaining an alternative derivation of its stationary solutions.  In Section~\ref{sec: asymptotic limit} we study the white noise limit for the generalized McKean-Vlasov dynamics, both for particle system and for the mean field limit. Finally, a brief summary of the results obtained in this paper and discussion on future work are presented in Section~\ref{sec: discussion}.

%\textcolor{red}{ what is the overall aim of this paper? We study a number of mean-field models. we show that they share common physical structure (GENERIC formalism) and mathematical properties (exact formula in quadratic case, characterisaton of invariant measures, number of invariant measures and   convergence to an invariant measure, effect of the noise, relationships among them (asymptotic behaviour)?}

%\textcolor{red}{More motivation??}
%\section{Derivation of the gLMV equation}

%\begin{itemize}

%\item McKean-Vlasov dynamics with colored noise.

%\item System of $N$ weakly interacting particles coupled to a heat bath.

%\item Pass to the limit: first eliminate the heat bath, then send $N \rightarrow +\infty$.

%\item Obtain the generalized McKean SDE (with memory).

%\item Markovian approximation.

%\end{itemize}

\section{Quadratic confining and interaction potentials: calculation of the fundamental solution, spectral analysis and convergence to equilibrium}
\label{sec:quadratic}

In this section, we compute and compare the spectrum of the three operators associated to the overdamped McKean-Vlasov dynamics, the underdamped McKean-Vlasov dynamics and the generalized McKean-Vlasov (gMV) dynamics. In particular, in Section~\ref{sec: particle systems} we study the finite dimensional problem, and in Section~\ref{subsec:mean-field} we calculate the spectrum and the fundamental solution of the mean field PDEs, for quadratic interaction and confining potentials.
%
%%%%%%%%%%%%%%%%%%%%%%
%
\subsection{Finitely many weakly interacting particles}
\label{sec: particle systems}
We consider the following $N$-particle weakly interacting particle systems where the dynamics of the  $i$-particle ($i=1,\ldots, N$) depends on the average interaction with all other particles.
\begin{itemize}
\item The overdamped McKean-Vlasov (oMV)  dynamics
\begin{equation}
\label{e: MV particles}
dX_i=-\nabla V(X_i)\,dt-\frac{1}{N}\sum_{j=1}^N \nabla U(X_i-X_j)\,dt+\sqrt{2\beta^{-1}}dW_i.
\end{equation}
\item The underdamped McKean-Vlasov  (uMV) dynamics 
\begin{subequations}
\label{e: VFP particles}
\begin{eqnarray}
dQ_i&=&P_i\,dt,
\\ dP_i&=&-\nabla V(Q_i)\, dt-\frac{1}{N}\sum_{j=1}^N \nabla U(Q_i-Q_j)\,dt-\gamma P_i\,dt+\sqrt{2\gamma\eta^{-1}} d W_i.
\end{eqnarray}
\end{subequations}
\item The generalized McKean-Vlasov (gMV) dynamics
\begin{subequations}
\label{eq: GLMV particles}
\begin{eqnarray}
dQ_i&=&P_i\,dt,
\\dP_i&=&-\nabla V(Q_i)\,dt-\frac{1}{N}\sum_{j=1}^N \nabla U(Q_i-Q_j)\,dt+\lambda^T Z_i \,dt,
\\ dZ_i&=&-P_i \lambda \,dt-A Z_i\,dt+\sqrt{2 \beta^{-1} A}dW_i.
\end{eqnarray}
\end{subequations}
%\item The McKean-Vlasov dynamics with coloured noise
%\begin{subequations}
%\begin{eqnarray}
%dX_i&=&-\nabla V(X_i)\,dt-\frac{1}{N}\sum_{j=1}^N \nabla U(X_i-X_j)\,dt+\delta n_i\,dt
%\\ dn_i&=&=-\alpha n_i \,dt +\sqrt{2\alpha\beta^{-1}} dW_i.
%\end{eqnarray}
%\end{subequations}
\end{itemize}
In this section we consider the case where both the external and interacting potentials are quadratic 
\begin{equation}
\label{eq: quadratic potentials}
V(\xi)=\frac{1}{2}\omega^2 \xi^2 \quad\text{and}\quad U(\xi)=\frac{1}{2}\eta^2 \xi^2.
\end{equation} 
In this case, all the dynamics can be written in the form of a (possibly degenerate) Ornstein-Ulenbeck process~\cite{Metafune2002,LorenziBertoldi2007,OPS2015}
\begin{equation}
d\Y^N =\B^N \Y^N\,dt +\sqrt{2\beta^{-1} \D^N}\,dW^N,
\end{equation}
with the corresponding state spaces, drift and diffusion matrices (written using block matrices):
\begin{itemize}
\item The quadratic oMV dynamics
\begin{align*}
&\Y^N_{oMV}=(X_1,\ldots, X_N)\in \R^{dN},\quad\Q^N_{oMV}=I_{dN\times dN};
\\&\B^N_{oMV}=\begin{pmatrix}
-\Big(\omega^2+\frac{N-1}{N}\eta^2\Big)&\ldots&\frac{1}{N}\eta^2\\
\vdots&\ddots&\vdots\\
\frac{1}{N}\eta^2&\ldots&-\Big(\omega^2+\frac{N-1}{N}\eta^2\Big)
\end{pmatrix}\in \R^{dN\times dN}.
\end{align*}
\item The quadratic uMV dynamics
\begin{align*}
& \Y^N_{uMV}=(Q_1,\ldots,Q_N,P_1,\ldots, P_N)\in\R^{2d N},\quad \B^N_{uMV}=\left(
\begin{array}{c|c}
0 & I \\ \hline
\B^N_{oMV} &-\gamma
\end{array}\right),\quad \D^N_{uMV}=\left(
\begin{array}{c|c}
0 & 0 \\ \hline
0 &\gamma
\end{array}\right).
\end{align*}
\item The quadratic gMV dynamics
\begin{align*}
& \Y^N_{gMV}=(Q_1,\ldots, Q_N, P_1,\ldots,P_N,Z^1_1,\ldots, Z_1^m,\ldots, Z_N^1,\ldots, Z_N^m)\in \R^{(2+m)dN},
\\&
\\&\B^N_{gMV}=\begin{pmatrix}
0&I&0&\ldots&0\\
\B^N_{oMV}&0&\lambda_1&\ldots&\lambda_m\\
0&-\lambda_1&-\alpha_1&\ldots&0\\
\vdots&\vdots&\vdots&\ddots&\vdots\\
0&-\lambda_m&0&\ldots&-\alpha_m
\end{pmatrix},\quad \D^N_{gMV}=\begin{pmatrix}
0&0&0&\ldots&0\\
0&0&0&\ldots&0\\
0&0&\alpha_1&\ldots&0\\
\vdots&\vdots&\vdots&\ddots&\vdots\\
0&0&0&\ldots&\alpha_m
\end{pmatrix}.
\end{align*}
%\item The quadratic MV dynamic with coloured noise
%\begin{align*}
%\Y^N_{cMV}=(X_1,\dots,X_N, n_1,\dots, n_N)\in \R^{2Nd}, \quad \B^N_{cMV}=\left(
%\begin{array}{c|c}
%B^N_{MV} & \delta I \\ \hline
%0 &-\alpha I
%\end{array}\right),\quad \D^N_{VFP}=\left(
%\begin{array}{c|c}
%0 & 0 \\ \hline
%0 &\alpha
%\end{array}\right).
%\end{align*}
\end{itemize}
\subsubsection{Calculation of the spectrum and of the fundamental solution}
The quadratic cases are of special interest because we can characterize the spectrum of the dynamics explicitly enabling us to compare their spectrum and rates of convergence to equilibrium. Our analysis will be based on the result in~\cite{Metafune2002,OPS2015} on the description of the spectrum of general hypoelliptic quadratic systems. Before introducing the result, we need to recall the relevant concept of hypoellipticity. We consider the following general Ornstein-Ulenbeck operator
\begin{equation}
\label{eq: OU operator}
P=Bx\cdot\nabla +\div(D\nabla ),\quad x\in\R^d,
\end{equation}
where $B$ and $D$ are real $d\times d$-matrices, with $D$ symmetric positive semi-definite. We say that $P$ is a hypoelliptic Ornstein-Uhlenbeck operator if one of the following equivalent conditions (thus all of them) holds
\begin{enumerate}[(i)]
\item The symmetric positive semidefinite matrices 
\begin{equation}
\label{eq: Dt}
D_t=\int_0^t e^{sB}D e^{sB^T}\,ds
\end{equation}
are nonsingular for some (equivalently, for all) $t>0$, i.e., $\det D_t>0$.
\item The Kalman rank condition holds
$$
\mathrm{Rank}[B|D^\frac{1}{2}]=d\quad\text\quad
$$
where $
[B|D^\frac{1}{2}]=[D^\frac{1}{2},BD^\frac{1}{2},\dots, B^{d-1}D^\frac{1}{2}]
$ is the $d\times d^2$ matrix obtained by writing consecutively the columns of the matrices $B^jD^\frac{1}{2}$ ($j=0,\dots, d-1$) with $D^\frac{1}{2}$ the symmetric positive semi-definite matrix given by the square root of $D$.
\item The H\"{o}rmander condition holds, i.e., the Lie algebra generated by $Y_0,X_1,\dots,X_d $ has full rank at every point in $\R^d$,
$$
\forall x\in\R^d,\quad \mathrm{Rank}\, \mathcal{L}(X_1,\dots, X_d,Y_0)(x)=d,
$$
with 
$$
Y_0:=Bx\cdot\nabla, \quad X_i=\sum_{j=1}^d D_{ij}\partial_{x_j}, i=1,\dots, d.
$$
\end{enumerate}
It is known that~\cite{Kol34,LorenziBertoldi2007} when the Ornstein-Ulenbeck operator $P$ in \eqref{eq: OU operator} is hypoelliptic, then it has a smooth fundamental solution $\Gamma(t,x,y)$ given explicitly by
\begin{equation}
\label{eq: fundamental sol}
\Gamma(t,x,y)=\frac{1}{(4\pi)^{d/2}\sqrt{\det D_t}}\exp\Big(-\frac{1}{4}(x-e^{-tB}y)^T D_t^{-1} (x-e^{-tB}y)\Big),
\end{equation}
where $D_t$ is given by Equation \eqref{eq: Dt}. Furthermore, it admits a unique invariant measure that is absolutely continuous with respect to the Lebesgue measure $d\mu(x)=\rho(x)\,dx$ with
$$
\rho(x)=\frac{1}{(4\pi)^{d/1}\sqrt{\det D_\infty}}\exp\Big(-\frac{1}{4}x^T D_\infty^{-1}x\Big),
$$
where
$$
D_\infty=\int_0^\infty e^{sB}D e^{sB^T}\,ds.
$$
We will make use of the following result.
\begin{proposition}[~\cite{Metafune2002,OPS2015}]
\label{prop: spectrum}
Suppose that
$$
P=Bx\cdot\nabla +\div(D\nabla),\quad x\in \R^d,
$$
is a hypoelliptic Ornstein-Ulenbeck operator that has the invariant measure $d\mu(x)=\mu(x)\,dx$. Then the spectrum of the operator $P: L^2_\mu\to L^2_\mu$ equipped with the domain
$$
D(P)=\{u\in L^2_\mu: Pu\in L^2_mu\},
$$
is only composed of eigenvalues with finite algebraic multiplicities given by
\begin{equation}
\label{eq: spectrum}
\sigma(P)=\Big\{\sum_{\lambda\in\sigma(B)}~\lambda k_\lambda:\, k_\lambda\in \mathbb{N}\Big\}.
\end{equation} 
\end{proposition}
We note that in~\cite{Metafune2002} the spectrum in all $L^p_\mu$ spaces is calculated. By direct verification, see for instance~\cite[Section 4.2]{OPS2012} for such a verification for similar models, one can show that all the quadratic systems considered in the previous section  are hypoelliptic. To characterize their spectrum, according to Proposition \ref{prop: spectrum}, we need to calculate eigenvalues of their corresponding drift matrices $B$.
\begin{itemize}
\item \textit{Spectrum of the quadratic oMV dynamics}.
The eigenvalues of $\B^N_{oMV}$ can be found by direct computation:
$$
\sigma(\B^N_{oMV})=\{-\omega^2,-(\omega^2+\eta^2)\}.
$$
Thus
$$
\sigma (L^N_{oMV})=\{-\omega^2 k_1-(\omega^2+\eta^2) k_2: k_1, k_2\in \mathbb{N}\}.
$$
\item \textit{Spectrum of the quadratic uMV dynamics}. Let $\alpha$ be an eigenvalue of $\B^N_{uMV}$, that is, there exist $\xi_1,\xi_2\in\R^{dN}$ such that
\begin{equation*}
\begin{pmatrix}
0&I\\
\B^N_{oMV}&-\gamma
\end{pmatrix}\begin{pmatrix}
\xi_1\\\xi_2
\end{pmatrix}=\alpha\begin{pmatrix}
\xi_1\\\ xi_2.
\end{pmatrix}
\end{equation*}
This implies that $\xi_2=\alpha \xi_1$ and $\B^N_{oMV}\xi_1-\gamma \xi_2=\lambda \xi_2$. Substituting the former into latter, we get 
\[
\B^N_{oMV}\xi_1=\alpha(\alpha+\gamma) \xi_1.
\]
Thus $\alpha(\alpha+\gamma)$ is an eigenvalue of $\B^N_{oMV}$, from which we deduce that
\[
\alpha(\alpha+\gamma)=-\omega^2 \quad\text{or}\quad \alpha(\alpha+\gamma)=-(\omega^2+\eta^2).
\]
Solving these equations we find
\begin{equation}
\alpha_{1,2}=\frac{-\gamma\pm \sqrt{\gamma^2-4\omega^2}}{2}\quad\text{and}\quad\alpha_{3,4}=\frac{-\gamma\pm \sqrt{\gamma^2-4(\omega^2+\eta^2)}}{2}.
\end{equation}
The spectrum of the quadratic uMV dynamics is then given by
\[
\sigma(L^N_{uMV})=\Big\{\sum_{i=1}^4 k_i\alpha_i: k_i\in \mathbb{N}, i=1,\ldots, 4\Big\}.
\]
\item \textit{Spectrum of the quadratic gMV dynamics}. Similarly, let $\nu$ be an eigenvalue of $\B^N_{gMV}$, that is, for some $\xi_1,\xi_2, \zeta_1, \ldots,\zeta_{m}\in\R^{dN}$, it holds
\[
\begin{pmatrix}
0&I&0&\ldots&0\\
\B^N_{oMV}&0&\lambda_1&\ldots&\lambda_m\\
0&-\lambda_1&-\alpha_1&\ldots&0\\
\vdots&\vdots&\vdots&\ddots&\vdots\\
0&-\lambda_m&0&\ldots&-\alpha_m
\end{pmatrix}\begin{pmatrix}
\xi_1\\ \xi_2\\ \zeta_1\\  \vdots \\ \zeta_{m}
\end{pmatrix}=\nu \begin{pmatrix}
\xi_1\\ \xi_2\\ \zeta_1\\  \vdots \\ \zeta_{m}
\end{pmatrix}.
\]
This is equivalent to the system
\begin{align*}
\begin{cases}
\xi_2&=\nu \xi_1,\\
B^N_{MV}\xi_1+\sum_{j=1}^m\lambda_j \zeta_j&=\nu\xi_2,\\
-\lambda_j \xi_2-\alpha_j\zeta_j&=\nu\zeta_j \quad\text{for all}\quad j=1,\ldots, m.
\end{cases}
\end{align*}
By substituting the first and the last $m$ equations into the second one, we get
\[
\B^N_{MV}\xi_1=\nu \Big(\nu+\sum_{j=1}^m\frac{\lambda_j^2}{\nu+\alpha_j}\Big)\xi_1.
\]
It follows that $\nu \Big(\nu+\sum_{j=1}^m\frac{\lambda_j^2}{\nu+\alpha_j}\Big)$ is an eigenvalue of $\B^N_{MV}$. Thus, we have
\begin{equation*}
\nu \Big(\nu+\sum_{j=1}^m\frac{\lambda_j^2}{\nu+\alpha_j}\Big)=-\omega^2\quad\text{or}\quad \nu \Big(\nu+\sum_{j=1}^m\frac{\lambda_j^2}{\nu+\alpha_j}\Big)=-(\omega^2+\eta^2).
\end{equation*}
In general these equations, which are equivalent to two polynomial equations of degree $m+2$, can not be solved explicitly. Suppose that $\nu_1,\ldots,\nu_{m+2}$ and $\nu_{m+3},\ldots, \nu_{2m+4}$ are solutions to the first and second equations, respectively. Then the spectrum of the quadratic gMV dynamics is given by
\[
\sigma(L^N_{gMV})=\Big\{\sum_{j=1}^{2m+4}\nu_j k_j: k_j\in \mathbb{N}, j=1,\ldots,2m+4\Big\}.
\]
\end{itemize}
\begin{remark}

From the calculation of the spectrum of the Fokker-Planck operator and by applying \cite[Theorem 2.7]{OPS2015} it follows that we have exponentially fast convergence to equilibrium in the $L^2$-space weighted by the invariant measure. This is also true in $L^1$ at least for the overdamped McKean-Vlasov equation, see~\cite{Tamura1987}.

\end{remark}

\subsection{The mean field PDE}
\label{subsec:mean-field}
The mean-field limits of the quadratic systems considered in Section \ref{sec: particle systems} can be written in a common form 
\begin{equation}
\label{eq: quadratic mean-field}
dX(t)=BX(t)\,dt+K(X(t)-\langle X(t)\rangle)\,dt+\sqrt{2 D^\frac{1}{2}}dW(t),
\end{equation}
for suitable matrices $B,K$ and $D$ with $D$ symmetric positive semidefinite. Here $\langle X(t)\rangle$ denotes the average of $X(t)$ with respect to its own law.

In this section, we will obtain the fundamental solution of the Fokker-Planck equation associated to the SDE \eqref{eq: quadratic mean-field} which is given by
\begin{equation}
\begin{cases}
\label{eq: general eqn}
\partial_t\rho(x,t)=-\div\Big(Bx\rho(x,t)\Big)-\div\Big(\int_{\R^n}K (x-x')\rho(x',t)\,dx'\, \rho(x,t)\Big)+\div \Big(D\nabla\rho(x,t)\Big),\\
\rho(x,0)=\delta(x-x_0).
\end{cases}
\end{equation}
\begin{proposition}
\label{prop: fund}
The solution of the initial value problem \eqref{eq: general eqn}, i.e. the Green's function of the Fokker-Planck operator, is given by,
\begin{equation}
\rho(x,t)=\frac{1}{(2\pi)^{n/2}\sqrt{\det \Qc(t)}}\exp\Big[-\frac{1}{2}\Qc^{-1}(t)(x-m(t)x_0)\cdot (x-m(t)x_0)\Big],
\end{equation}	
where the mean vector and the covariance matrix are given, respectively, by
\begin{equation}
m(t)=e^{tB}\quad\text{and}\quad \Qc(t)=2\int_0^t e^{2s(B+K)} D\,ds.
\end{equation}
\end{proposition}
\begin{proof}
Equation~\eqref{eq: general eqn} is a nonlinear nonlocal degenerate partial differential equation. Since the law of the finite particle system corresponding to \eqref{eq: quadratic mean-field} when starting at a deterministic initial condition is Gaussian, its meanfield limit is also Gaussian. Therefore, we look for the fundamental solution to~\eqref{eq: general eqn} of the form of a multivariate Gaussian distribution with time-dependent mean and covariance matrix
\begin{equation}
\label{eq: Gaussian}
\rho(x,t)=\frac{1}{(2\pi)^{n/2}\sqrt{\det \Qc(t)}}\exp\Big[-\frac{1}{2}\Qc^{-1}(t)(x-m(t)x_0)\cdot (x-m(t)x_0)\Big].
\end{equation}  
We will find the mean $m(t)$ and the variance $\Qc(t)$ by requiring that $\rho(x,t)$ satisfies Equation~\eqref{eq: general eqn}. Let us define
\begin{equation}
\label{eq: g}
g(x,t):=-\log\rho(x,t)=\frac{1}{2}\Qc^{-1}(t)(x-m(t)x_0)\cdot (x-m(t)x_0)+s(t),
\end{equation}
where $s(t)=\frac{1}{2}\log\det\Qc(t)+\frac{n}{2}\log (2\pi)$. The logarithmic transformation has been used previously in the literature, particularly in optimal control \cite{Fleming1982}.  Substituting $\rho=\exp(-g)$ in \eqref{eq: general eqn}, we can rewrite the equation in terms of $g$ as follows
\begin{equation}
\label{eq: eqn g}
\partial_t g=-(B+K)x\cdot\nabla g+D:\nabla^2 g-D\nabla g\cdot\nabla g+h(t)\cdot\nabla g+\Tr(B+K),
\end{equation}
where
\begin{align*}
h(t)&=\int Kx'\rho(x',t)\,dx'
\\&=\frac{1}{(2\pi)^{n/2}\sqrt{\det \Qc(t)}}\int Kx'\exp\Big[-\frac{1}{2}\Qc^{-1}(t)(x'-m(t)x_0)\cdot (x'-m(t)x_0)\Big]\,dx'
\\&=Km(t)x_0.
\end{align*}
To find $\Qc$ and $m$ we set equal the corresponding terms in two sides of \eqref{eq: eqn g}. Using the explicit formula  \eqref{eq: g} of $g$, we compute the left-hand side of \eqref{eq: eqn g}, which is the time-derivative of $g$
\begin{multline}
\partial_t g=\frac{1}{2}\partial_t (\Qc^{-1}(t))x\cdot x-\partial_t (\Qc^{-1}(t)m(t))x_0\cdot x+\frac{1}{2}\partial_t\Big(m(t)^TQ^{-1}(t)m(t)\Big)x_0\cdot x_0+\partial_t s(t).
\label{eq: LHS}
\end{multline}
Since $\nabla g=\Qc^{-1}(t)(x-m(t)x_0),\quad \nabla^2 g=\Qc^{-1}(t)$, the right-hand side of \eqref{eq: eqn g} is given by
\begin{align}
&-(B+K)x\cdot\nabla g+D:\nabla^2 g-D\nabla g\cdot\nabla g+h(t)\cdot\nabla g+\Tr(B+K)\notag
\\\qquad&=-\Qc^{-1}(t)(B+K)x\cdot (x-m(t)x_0)-\Qc^{-1}(t)D\Qc^{-1}(t)(x-m(t)x_0)\cdot (x-m(t)x_0)\notag
\\&\qquad\qquad+\Qc^{-1}(t)Km(t)x_0\cdot (x-m(t)x_0)+D:\Qc^{-1}(t)-\Tr(B+K)\notag
\\ \qquad&=\Big[-\Qc^{-1}(t)(B+K)-\Qc^{-1}(t)D\Qc^{-1}(t)\Big]x\cdot x-\Big[-m(t)^T\Qc^{-1}(t)(B+K)\notag
\\& \qquad-2m(t)^T\Qc^{-1}(t)D\Qc^{-1}m(t)-\Qc^{-1}Km(t)\Big]x_0\cdot x\notag
\\&\qquad-\Big[m(t)^T\Qc^{-1}(t)D\Qc^{-1}(t) m(t)+m(t)^T\Qc^{-1}(t)Km(t)\Big]x_0\cdot x_0+D:\Qc^{-1}(t)+\Tr(B+K).
\label{eq: RHS}
\end{align}
By comparing \eqref{eq: LHS} and \eqref{eq: RHS} we obtain the following equations
\begin{align}
&\partial_t \Qc^{-1}(t)=2\Big[-\Qc^{-1}(t)(B+K)-\Qc^{-1}(t)D\Qc^{-1}(t)\Big],\label{eq: Q}
\\&\partial_t(\Qc^{-1}(t)m(t))=-m(t)^T\Qc^{-1}(t)(B+K)-2m(t)^T\Qc^{-1}(t)D\Qc^{-1}m(t)-\Qc^{-1}Km(t),\label{eq: m}
\\& \partial_t(m(t)^TQ^{-1}(t)m(t))=-2\Big[m(t)^T\Qc^{-1}(t)D\Qc^{-1}(t) m(t)+m(t)^T\Qc^{-1}(t)Km(t)\Big],\label{eq: Qm}
\\&\partial_t s(t)=D:\Qc^{-1}(t)+\Tr(B+K).\label{eq: s}
\end{align}
We will solve these equations together with the initial value datum $\Qc(0)=\mathbf{0}, m(0)=\mathbf{I}$ and $s(0)=\frac{n}{2}\log(2\pi)$.
Equation \eqref{eq: Q} is a nonlinear Riccati-type equation for $\Qc^{-1}(t)$. We use the following formula for the derivative of the inverse of a matrix
\[
\frac{d A^{-1}(t)}{dt}=-A^{-1}(t)\frac{d A(t)}{dt}A^{-1}(t).
\] 
Applying this formula for $\Qc(t)$ and using \eqref{eq: Q} we get
\begin{align}
\partial_t \Qc(t)&=-\Qc(t)\partial_t (\Qc^{-1}(t))\Qc(t)=-2\Qc(t)\Big[-\Qc^{-1}(t)(B+K)-\Qc^{-1}(t)D\Qc^{-1}(t)\Big]\Qc(t)\notag
\\&=2 \Big[D+(B+K)\Qc(t)\Big].
\label{eq: derivative Q}
\end{align}
Using Duhamel's formula we obtain
\begin{equation}
\Qc(t)=2\int_0^t e^{2(t-s)(B+K)}D  \,ds= 2\int_0^t e^{2s(B+K)}D  \,ds.
\end{equation}
From \eqref{eq: Q} and \eqref{eq: m} we deduce that 
\begin{equation}
\label{eq: m2}
\partial_t m(t)=Bm(t),
\end{equation}
which yields $m(t)=e^{tB}$.
From \eqref{eq: Q} and \eqref{eq: m2}, we have
\begin{align*}
\partial_t(m(t)^T\Qc^{-1}(t)m(t))&=\partial_t m(t)^T\Qc^{-1}(t)m(t)+m(t)^T\partial_t(\Qc^{-1}(t))m(t)+m(t)^T\Qc^{-1}(t)\partial_t m(t)
\\&=-2\Big[m(t)^T\Qc^{-1}(t)D\Qc^{-1}(t) m(t)+m(t)^T\Qc^{-1}(t)Km(t)\Big],
\end{align*}
which is exactly \eqref{eq: Qm}. In other words, \eqref{eq: Qm} is a consequence of \eqref{eq: Q} and \eqref{eq: m}. Similarly we show now that \eqref{eq: s} is also redundant as a consequence of \eqref{eq: derivative Q}. In fact, applying the formula for the derivative of the determinant of a matrix we get
\begin{align*}
\partial_t s(t)&=\frac{1}{2}\partial_t \log\det\Qc(t)=\frac{1}{2}\frac{\partial_t\det\Qc(t)}{\Qc(t)}
\\&=\frac{1}{2}\Tr\Big(\Qc^{-1}(t)\partial_t \Qc(t)\Big)
\\&\overset{\eqref{eq: derivative Q}}{=}\Tr\Big[\Qc^{-1}(t)\big(D+(B+K)\Qc(t)\big)\Big]
\\&=D:\Qc^{-1}(t)+\Tr(B+K),
\end{align*}
which is indeed \eqref{eq: s}, as claimed. In conclusion, we find that
\begin{equation*}
\Qc(t)=2\int_0^t e^{2s(B+K)}D\,ds\quad \text{and}\quad m(t)=e^{tB}.
\end{equation*}
\end{proof}
\begin{remark}\ \\
1) By taking $d=1$, $B=-\gamma$, $K=\kappa$ and $D=Q$  we recover the known result~\cite[Equation (3.166)]{Frank} for the Shimizu-Yamada model (the overdamped McKean-Vlasov equation with quadratic interaction and confining potentials).

\noindent 2) Suppose that $\L=\S+\A$ is an infinitesimal generator of a diffusion process that possesses an invariant measure $\rho_\infty$ and satisfies 
\[
\S=\S^\ast,\quad \A=-\A^\ast~~\text{in}~~L^2_{\rho_\infty}~~\text{and}~~\S\rho_\infty=0=\A\rho_\infty.
\]
Then the following relation holds~\cite[Chapter 4]{Pavl2014}:
\begin{equation}
\label{eq: dual sols}
\L^\ast(f\rho_\infty)=(\S f-\A f)\rho_\infty,
\end{equation} 
From this relation one can find the fundamental solution to the forward Fokker-Plank equation, $\L^\ast\rho=0$, from that of the backward Fokker-Planck equation, $\L f=0$, and vice versa. 
\end{remark}

%\subsection{Exponential convergence to equilibrium}
%
%%%%%%%%%%%%%%%%%%%%%%%%%%%%%%%%%%%%%%%%%%%%%%%%%
%
\section{Nonconvex potentials: stationary solutions, phase transition and convergence to equilibrium}
\label{sec: nonconvex}
For non-convex confining potentials and for the Curie-Weiss quadratic interaction potential, it is well known that the the McKean-Vlasov equation exhibits a continuous phase transition~\cite{Dawson1983}: at sufficiently high temperatures only one stationary distribution exists, corresponding to zero mean (magnetization order parameter), whereas at low temperatures this mean zero state becomes unstable and two new branches emerge (for the Landau quadratic potential), corresponding to a nonzero magnetization. A natural question is whether the addition of inertia, i.e. the underdamped McKean-Vlasov equation, or the addition of memory in the system can change the structure of this pitchfork bifurcation. This problem was studied for the underdamped McKean-Vlasov equation in~\cite{DuongTugaut2016}, where it was shown that the presence of inertia does not change the bifurcation diagram. In this section we show that this is the case also for the generalized McKean-Vlasov dynamics. 

In this section, we make the following assumptions on the confining potential $V$, the interaction potential $U$ and the initial data $\rho_0$. They are made in order to rely on the results in~\cite{Tugaut13a,Tugaut13b, DuongTugaut2018}.
\begin{assumption}
\label{asspt}
\begin{enumerate}[(i)]
\item $V\in C^\infty(\R^d)$; $V(q)\geq C_4|q|^4-C_2 |q|^2$ for some positive constants $C_2$ and $C_4$; $|\nabla V(q)|\leq K|q|^{2m}$ for some positive constants $m$ and $K$; $\nabla^2V(q)>0$ when $q\notin \mathcal{K}$ for some compact subset $\mathcal{K}\subset\R^d$. Finally, $\displaystyle\lim_{|q|\to+\infty}\nabla^2V(q)=+\infty$.
\item $U$ is nonnegative and there exists an even, positive and convex polynomial function $G$ with $\deg(G)=:2n\geq 2$ such that $U(q)=G(|q|)$. 
\item The initial measure $\rho_0$ admits a $\mathcal{C}^\infty$-continuous density with respect to the Lebesgue measure and has a finite entropy. Furthermore, its $8r^2$-th moment with respect to the variable $q$, where $r:=\max\left\{m,n\right\}$, and its second moment with respect to the variable $p$ are finite.
\end{enumerate}
\end{assumption}
\subsection{Stationary solutions: characterization and phase transition}
We first provide a characterization of stationary solutions to Equation~\eqref{eq: GLMV}, i.e., solution to the following equation
\begin{equation}
\label{eq:stationary eqn}
\K[\rho]\rho=0,
\end{equation}
where
\begin{eqnarray}
\label{eq:stationary oper}
\K[\mu]\rho  & = & -\div_q(p\rho)+\div_p\Big[(\nabla_q V(q)+\nabla_q U\ast\mu (q)-\lambda^T z)\rho\Big]
\nonumber\\ &&+
\div_z\Big[( p \lambda +Az)\rho\Big]+\beta^{-1}\div_z(A\nabla_z\rho),
\end{eqnarray}
for a given $\mu\in L^1(\X)$. Note that for a given $\mu$, the operator $\K[\mu](\rho)$ is linear in $\rho$ which can be seen as the linearization of the (quadratic) operator $\K[\rho](\rho)$ around the state $\mu$.

We start by showing that the stationary McKean-Vlasov equation~\eqref{eq:stationary eqn} can be rewritten as a nonlinear integral equation. This reformulation of the stationary Fokker-Planck equation is well known for the overdamped McKean-Vlasov equation, see, e.g.~\cite{Tamura1984}[Lem. 4.1, Thm. 4.1].

\begin{proposition} 
\label{prop: characterization}
If there exists a solution $\rho_\infty\in L^1(\X)\cap L^\infty(\X)$ to Equation~\eqref{eq:stationary eqn} then
\begin{equation}
\label{eq: invariant}
\rho_\infty(q,p,z)=\frac{1}{Z}\exp\left[-\beta\,\Big(\frac{|p|^2}{2}+\frac{\|z\|^2}{2}+V(q)+(U\ast\rho_\infty)(q)\Big) \right],
\end{equation}
where $Z$ is the normalization constant
\begin{equation}
\label{eq: Z}
Z=\int_{\X}\exp\left[-\beta\,\Big(\frac{|p|^2}{2}+\frac{\|z\|^2}{2}+V(q)+(U\ast\rho_\infty )(q)   \Big) \right]\,d\x.
\end{equation}
Conversely, any probability measure whose density satisfies~\eqref{eq: invariant} is a stationary solution to~\eqref{eq: GLMV}.
\end{proposition}
\begin{proof}
The converse statement is obviously verified by direct computations. We adapt the proof of a similar statement for the underdamped McKean-Vlasov equation from~\cite[Proposition 1]{DuongTugaut2016} and \cite[Proposition 1\& 2]{Dressler1987} to prove the necessary condition for the generalized McKean-Vlasov equation . It consists of two steps. Due to the presence of the additional variable $z$, Step 1 below will be more involved than those in aforementioned papers.

\noindent \textbf{Step 1: The linearized equation}. We first consider the linearized equation
\begin{equation}
\label{eq: linearized eqn}
\K\rho:=\K[\mu]\rho=0,
\end{equation}
where $\mu\in L^1(\R^d\times\R^d\times\R^{dm})$ is given. We define
\begin{subequations}
\begin{eqnarray}
\label{eq: H and J}
\H_\mu(q,p,z)&:=&\frac{1}{2}|p|^2+V(q)+\frac{1}{2}\|z\|^2+(U\ast\mu)(q)  \quad \text{and}
\\ 
\J & =&  \begin{pmatrix}
0&-I&0\\
I&0&-\lambda^T\\
0&\lambda&0
\end{pmatrix}\in\R^{(2d+m)\times (2d+m)}.
\end{eqnarray}
\end{subequations}
Then Equation~\eqref{eq: linearized eqn} can be written in a compact form as
\[
\K[\mu]\rho=\div[\J\nabla\H_\mu \rho]+\div_z[A(z\rho+\beta^{-1}\nabla_z\rho)]=0,
\]
where the $\div$ and $\nabla$ in the first term are with respect to the full variable $\x=(q,p,z)$.

We prove the following claim: define
\[
f(q,p,z):=\frac{1}{\bar{Z}}\exp\left[-\beta\,\Big(\frac{|p|^2}{2}+\frac{\|z\|^2}{2}+V(q)+U\ast\mu(q)\Big) \right]=\frac{1}{\bar{Z}}\exp(-\beta \H_\mu),
\]
where $\bar{Z}$ is the normalization constant, $\bar{Z}=\int_{\mathbf{X}}e^{-\beta \H_\mu}\,dx$, and
\begin{align*}
\A:=&\Big\{u: \X\to\R: g:= u f^{-1/2}~~\text{satisfies that}~~g\in C^1, \|\J\nabla \H_{\mu}g^2\|_1<\infty, \|\div(\J\nabla\H_\mu g^2)\|_1<\infty
\\&\qquad \|f^{1/2}g A\nabla_z(f^{-1/2}g)\|_{1}<\infty ~~\&~~\|\div (f^{1/2}g A\nabla_z(f^{-1/2}g))\|_1<\infty\Big\}.
\end{align*}
Then $f$ is the unique solution in $\A$ to the linearized equation~\eqref{eq: linearized eqn}. The conditions in $\A$ will ensure later the application of the following divergence theorem in the whole space~\cite[Section 4.5.2]{giga2010}: 
\begin{theorem}
\label{thm: div theorem}
Let $F$ be a $C^1$ vector field in $\R^n$. Assume that $\|F\|_1<\infty$ and $\|\div F\|_1<\infty$. Then
\begin{equation*}
\int_{\R^n}\div F\,d\x=0.
\end{equation*}
\end{theorem}
To verify $f\in\A$ it is sufficient to show that $\|\J\nabla \H f\|_1<\infty$. This is true under Assumption~\ref{asspt} on the confining and interaction potentials. In addition, straightforward computations give
\[
\div[\J\nabla\H_\mu f]=\J\nabla\H_\mu \cdot \nabla f+f\div(\J\nabla\H_\mu)=-\beta \J\nabla\H_\mu\cdot\nabla\H f=0 \quad\text{and}\quad \div_z[A (zf+\beta^{-1}\nabla_z f]=0,
\]
where we have used the antisymmetry of $\J$ and the fact that $\nabla_z f=-\beta z f$.
Therefore, $\K[\mu] f=0$. Suppose that there exists another solution $\bar{f}\in\A$ to Equation~\eqref{eq: linearized eqn}. We need to prove that $\bar{f}=f$. Let $g:=\bar{f} f^{-1/2}$, i.e., $\bar{f}=g f^{1/2}$. We have
\begin{align*}
\div[\J\nabla\H_\mu \bar{f}]=\J\nabla\H_\mu \cdot \nabla(gf^{1/2})=f^{1/2}\J\nabla\H_\mu \cdot \nabla g+g\J\nabla\H_\mu \cdot \nabla f^{1/2}=f^{1/2}\J\nabla\H_\mu \cdot \nabla g=f^{1/2}\div[\J\nabla\H_\mu g],
\end{align*}
where we have used the fact that $\J\H_\mu\cdot\nabla f^{1/2}=-\frac{\beta}{2}f^{1/2} \J\H_\mu\cdot\nabla\H_\mu =0$. Furthermore,
\begin{align*}
\div_z[A(z \bar{f}+\beta^{-1}\nabla_z\bar{f})]&=\div_z[A\Big(z \bar{f}+\beta^{-1}\nabla_z(f f^{-1/2}g))]
\\&=\div_z\Big[A(z \bar{f}+\beta^{-1}f\nabla_z(f^{-1/2}g)+\beta^{-1}f^{-1/2}g\nabla_z f\Big)\Big]
\\&=\beta^{-1}\div_z\Big[A f\nabla_z(f^{-1/2}g)\Big],
\end{align*}
since $z\bar{f}+\beta^{-1}gf^{-1/2}\nabla_zf=z\bar{f}-gf^{-1/2}fz=z\bar{f}-z\bar{f}=0$. Therefore, we obtain that 
\[\K[\mu]\bar{f}=f^{1/2}\div[J\nabla\H_\mu g]+\beta^{-1}\div_z\Big[A f\nabla_z(f^{-1/2}g)\Big].
\]
We now define 
\[
\Q g:=-f^{-1/2}\K[\mu]\bar{f}=-\div[J\nabla\H_\mu g]-\beta^{-1}f^{-1/2}\div_z\Big[A f\nabla_z(f^{-1/2}g)\Big].
\]
We have $Qg=0$. On the other hand, we have
\begin{align*}
0&=\langle \Q g,g\rangle_{L^2}=-\int_{\X} \Big(\div[\J\nabla\H_\mu g]+\beta^{-1}f^{-1/2}\div_z\Big[A f\nabla_z(f^{-1/2}g)\Big]\Big) g\,d\x
\\&=\frac{1}{2}\int_{\X} \div[\J\nabla\H_\mu g^2]\,d\x+\int_{\X}\div[f^{1/2}g A\nabla_z(f^{-1/2}g)]\,d\x+\beta^{-1}\int_{\X}f A\nabla_z(f^{-1/2}g)\cdot\nabla_z(f^{-1/2}g)\,d\x
\\&=\beta^{-1}\int_{\X} f A\nabla_z(f^{-1/2}g)\cdot\nabla_z(f^{-1/2}g)\,d\x.
\end{align*}
We note that the first two integrals vanish because $f\in \A$ and Theorem~\ref{thm: div theorem}. Since the matrix $A$ is positive definite, it follows that $\nabla_z(f^{-1/2}g)=0$, i.e., $f^{-1/2}g=h(q,p)$ for some function $h$. Hence $\bar{f}=gf^{1/2}=f h(q,p)$ and
\[
0=\K[\mu]\bar{f}=\K[\mu](f h)=\div[\J\nabla\H_\mu(fh)]+\div_z[A(z fh+\beta^{-1}\nabla_z(fh))]=f \J\nabla\H_\mu \cdot\nabla h+h\K[\mu]f.
\]
This implies that 
\[0=\J\nabla\H_\mu\cdot\nabla h=\nabla_p\H_\mu\cdot\nabla_q h+(-\nabla_q \H_\mu+\lambda\nabla_z \H_\mu)\cdot\nabla_p h,
\]
hence $\nabla_p h=\nabla_q h=0$, i.e., $h$ is a constant. Since $\|\bar{f}\|_{L^1}=\|f\|_{L^1}=1$, $h$ must be $1$ and therefore $\bar{f}=f$ as required.

\noindent \textbf{Step 2: the nonlinear equation}. Now suppose that $\rho_\infty$ satisfies the (nonlinear) stationary equation~\eqref{eq:stationary eqn}, i.e., $\K[\rho_\infty]\rho_\infty=0$. According to Step 1, $\rho_\infty$ must satisfy
\begin{equation}
\rho_\infty=\frac{\exp\left[-\beta\,\Big(\frac{|p|^2}{2}+\frac{\|z\|^2}{2}+V(q)+U\ast\rho_\infty(q)\Big) \right]}{\int_{\X} \exp\left[-\beta\,\Big(\frac{|p|^2}{2}+\frac{\|z\|^2}{2}+V(q)+U\ast\rho_\infty(q)\Big) \right]\,d\x},
\end{equation}
which is the claimed statement. 
\end{proof}
As a direct consequence of Proposition~\ref{prop: characterization} we obtain the following corollary, which is a natural generalization of the fact that the invariant measures of the finite dimensional underdamped Langevin and generalized Langevin dynamics are product measures.
\begin{corollary}
\label{cor: stationary}
The number of stationary states of the overdamped, underdamped and generalized Mckean-Vlasov equations is the same.
\end{corollary}
\begin{proof}
Stationary solutions of the overdamped, the underdamped and the generalized McKean-Vlasov equations satisfy the following integral equations
\begin{align*}
\rho^{oMV}_\infty&=\frac{\exp\left[-\beta\,\Big(V(q)+U\ast\rho_\infty(q)\Big) \right]}{\int_{\X} \exp\left[-\beta\,\Big(V(q)+U\ast\rho_\infty(q)\Big) \right]\,dq},
\\\rho^{uMV}_\infty&=\frac{\exp\left[-\beta\,\Big(\frac{|p|^2}{2}+V(q)+U\ast\rho_\infty(q)\Big) \right]}{\int_{\X} \exp\left[-\beta\,\Big(\frac{|p|^2}{2}+V(q)+U\ast\rho_\infty(q)\Big) \right]\,d\x}=\frac{\exp\Big(-\frac{|p|^2}{2}\Big)}{\int_{\R^d}\exp\Big(-\frac{|p|^2}{2}\Big)\,dp}\rho^{MV}_\infty,
\end{align*}
and
\begin{equation*}
\rho^{gMV}_\infty=\frac{\exp\left[-\beta\,\Big(\frac{|p|^2}{2}+\frac{\|z\|^2}{2}+V(q)+U\ast\rho_\infty(q)\Big) \right]}{\int_{\X} \exp\left[-\beta\,\Big(\frac{|p|^2}{2}+\frac{\|z\|^2}{2}+V(q)+U\ast\rho_\infty(q)\Big) \right]\,d\x}=\frac{\exp\Big(-\frac{\|z\|^2}{2}\Big)}{\int_{\R^{md}}\exp\Big(-\frac{\|z\|^2}{2}\Big)\,dz}\frac{\exp\Big(-\frac{|p|^2}{2}\Big)}{\int_{\R^d}\exp\Big(-\frac{|p|^2}{2}\Big)\,dp}\, \rho^{MV}_\infty,
\end{equation*}
respectively.
Due to the separability of the variables $p$ and $z$ in the above formulas, the number of stationary solutions of the underdamped the generalized McKean-Vlasov equations are the same as that of the overdamped Mckean-Vlasov equation.
\end{proof}

\begin{remark}
A consequence of Corollary~\ref{cor: stationary} is that, for nonconvex confining potentials, and for the Curie-Weiss quadratic interaction potential, the bifurcation diagrams and even the critical temperature--see, e.g.~\cite{Dawson1983}[Result II, Remark 3.3.1]--are the same for the overdamped, underdamped and generalized McKean-Vlasov dynamics. 
\end{remark}
\begin{remark}
(Effect of colored noise on the structure of phase transitions)

The McKean-Vlasov dynamics with colored noise is given by
\begin{subequations}
\label{eq: cMV-SDE}
\begin{eqnarray}
dX(t)&=&-\nabla V(X(t))\,dt-\nabla U(X(t))\ast\rho_t\,dt+\delta n\,dt,
\\ dn(t)&=&=\alpha n(t)\,dt+\sqrt{2\alpha\beta^{-1}}dW(t),
\end{eqnarray}
where $\rho_t=\mathrm{Law}(X(t),n(t))$, $\delta,\alpha$ are given positive constants.
\end{subequations}
The Fokker-Planck equation associated to the SDE \eqref{eq: cMV-SDE} is given by
\begin{equation}
\label{eq: cMV}
\frac{\partial \rho}{\partial t}=\div_x\Big[(\nabla V(x)+\nabla U(q)\ast\rho+\delta n)\rho\Big]+\alpha\Big(\div_n(n\rho)+\beta^{-1}\Delta_n \rho\Big).
\end{equation}
There does not appear to be an obvious gradient structure for Equation~\eqref{eq: cMV}. An invariant measure $\rho_\infty=\rho_\infty(x,n)\,dxdn$ to \eqref{eq: cMV} satisfies
$$
\div_x\Big[(\nabla V(x)+\nabla U(x)\ast\rho_\infty+\delta n)\rho_\infty\Big]+\alpha\Big(\div_n(n\rho_\infty)+\beta^{-1}\Delta_n \rho_\infty\Big)=0.
$$
It implies that $\rho_\infty$ can not be written in the form
$$
\rho_\infty(x,z)=Z^{-1}\exp\bigg[-\beta\Big(V(x)+U\ast\rho_\infty(x)+\frac{\|n\|^2}{2}\Big)\bigg],
$$
For this reason, the effect of colored noise on the structure of phase transitions turns out to be nontrivial. We will study this problem in a forthcoming work.
\end{remark}

\subsection{Convergence to equilibrium}
In this section, we prove that solutions to the generalized McKean-Vlasov equation converge to a stationary solution using the free energy approach developed in~\cite{BonillaCarrilloSoler97}, see also \cite{DuongTugaut2018}. We note that in this paper we do not address the problem of well-posedness and regularity of solutions to the generalized McKean-Vlasov equation. Under appropriate assumptions on the initial data, it should be possible, in principle, to adapt the techniques of~\cite{BonillaCarrilloSoler97, BouchutDolbeault95}
to show that a solution to the generalised McKean-Vlasov equation is sufficiently smooth so that the computations in this section can be fully justified.

We consider the following \textit{free energy functional}
\begin{align}
\label{eq: free energy}
F(\rho)&=\beta^{-1}\int \rho\log\rho + \int \Big[\frac{1}{2}|p|^2+V(q)+\frac{1}{2}\|z\|^2\Big]\rho 
\nonumber \\&+\frac{1}{2}\int U(q-q')\rho(q',p',z')\rho(q,p,z)\,dq'dp'dz'\,dqdpdz
\nonumber
\\&=\int \Big[\frac{1}{2}|p|^2+V(q)+\frac{1}{2}\|z\|^2+\frac{1}{2}U\ast\rho+\beta^{-1}\log\rho\Big]\rho,
\end{align}
and define $h(t):=F(\rho(\cdot, t))$.

The following lemma establishes useful properties of the free energy functional $F(\cdot)$ and of its time derivative along the flow $\rho(\cdot)$.
\begin{lemma}
\label{lem: properties of free energy}
 The following assertions hold
\begin{enumerate}[(1)]
\item $F$ is bounded from below.
\item $h$ is a decreasing function of time and its time derivative is given by
\begin{equation}
\label{eq: time derivative of free energy}
\dot{h}(t)=-\int A(z\sqrt{\rho}+2\beta^{-1}\nabla_z\sqrt{\rho})\cdot (z\sqrt{\rho}+2\beta^{-1}\nabla_z\sqrt{\rho})\leq 0.
\end{equation}
\item The limit $\lim\limits_{t\to+\infty} h(t)$ is well-defined.
\end{enumerate}
\end{lemma}
\begin{proof}
We start with proving that the free energy functional is bounded from below. In fact, since $U\geq 0$ and $\inf_{q\in\R^d}\Big(V(q)-\frac{|q|^2}{2}\Big)\geq C_1>-\infty$ we have
\begin{align*}
F(\rho)&\geq C_1+\int \Big[\frac{|p|^2}{2}+\frac{|q|^2}{2}+\frac{\|z\|^2}{2}\Big]\rho+\beta^{-1}\int \rho\log\rho
\\&=C_1+\beta^{-1}\mathcal{H}(\rho||\mu)-C_2
\\&\geq C_1-C_2>-\infty,
\end{align*}
where $\mu:=\frac{1}{Z_1}\exp\Big(-\beta\big(\frac{|p|^2}{2}+\frac{|q|^2}{2}+\frac{\|z\|^2}{2}\big)\Big)$ with $Z_1$ being the normalization constant, $C_2=\ln Z_1$, and $\mathcal{H}(\rho||\mu)\geq 0$ is the relative entropy between $\rho(\x)\,dx$ and $\mu$ which is defined by
\[
\mathcal{H}(\rho||\mu):=\int \log\Big(\frac{\rho(\x)}{\mu(\x)}\Big)\rho(\x)\,d\x.
\]
It is well known that $\mathcal{H}(\rho||\mu)\geq 0$, see for instance~\cite[Lemma 1.4.1]{DupuisEllis97}. Next we prove the second assertion. Using $\H$ and $\J$ defined in \eqref{eq: H and J}, Equation~\eqref{eq: GLMV} can be written as
\begin{equation*}
\partial_t\rho=-\div\Big[\J\nabla \H\rho\Big]+\div_z\Big[A (z\rho+\beta^{-1}\nabla_z\rho)\Big],
\end{equation*}
where $\div$ on the right hand side of the above equation denotes the divergence with respect to $\x=(q,p,z)$. This representation is more convenient for the following computations:
\begin{align*}
\dot{h}(t)&=\frac{d}{dt}F(\rho(t))
\\&=\int \Big[\H+\beta^{-1}(\log\rho+1)]\partial_t\rho
\\&=\int \Big[\H+\beta^{-1}\log\rho\Big]
\cdot\Big[-\div(\J\nabla\H\rho)+\div_z(Az\rho+\beta^{-1}A\nabla_z\rho)\Big]
\\&=\int \Big[\J\H\cdot \H\rho+\beta^{-1}\J\nabla \H\cdot\nabla \rho\Big]-(Az\rho+\beta^{-1}A\nabla_z\rho)\cdot\nabla_z(H+\beta^{-1}\log\rho)
\\&\overset{(*)}{=}-\int \frac{1}{\rho} A(z\rho+\beta^{-1}\nabla_z
\rho)\cdot (z\rho+\beta^{-1}\nabla_z\rho)
\\&=-\int A(z\sqrt{\rho}+2\beta^{-1}\nabla_z\sqrt{\rho})\cdot (z\sqrt{\rho}+2\beta^{-1}\nabla_z\sqrt{\rho}),
\end{align*}
where to obtain the equality $(*)$ we have used the fact that the square-bracket in the line above vanishes due to the antisymmetry of $\J$ and integration by parts while the last equality follows from $(*)$ using the property that $\nabla_z\sqrt{\rho}=\frac{\nabla_z\rho}{2\sqrt{\rho}}$. Finally, the last assertion is a direct consequence of the first and second ones.
\end{proof}

\begin{corollary}For any $T>0$, we have
\begin{equation}
\label{eq: limit of derivative free energy}
\lim_{t\to\infty}\int_0^T\int_{\X}A\big(z\sqrt{\rho^{(t)}}+2\beta^{-1}\nabla_z\sqrt{\rho^{(t)}}\big)\cdot \big(z\sqrt{\rho^{(t)}}+2\beta^{-1}\nabla_z\sqrt{\rho^{(t)}}\big)\,d\x\,ds=0
\end{equation}
where $\rho^{(t)}$ is defined by $\rho^{(t)}(s,\x):=\rho(t+s,\x)$.
\end{corollary}
\begin{proof}
For any given $T>0$, by part (iii) of Lemma~\ref{lem: properties of free energy}, we have
\[
\lim_{t\to\infty}\int_0^T \dot{h}(t+s)\,ds=\lim_{t\to\infty} (h(t+T)-h(t))=0.
\]
On the other hand, by part (ii), we get
\begin{align*}
0=\lim_{t\to\infty}\int_0^T \dot{h}(t+s)\,ds&=\lim_{t\to\infty}\int_0^T\int_{\X}A\big(z\sqrt{\rho^{(t)}}+2\beta^{-1}\nabla_z\sqrt{\rho^{(t)}}\big)\cdot \big(z\sqrt{\rho^{(t)}}+2\beta^{-1}\nabla_z\sqrt{\rho^{(t)}}\big)\,d\x\,ds
\end{align*}
which proves the statement of the corollary.
\end{proof}
Given any sequence $\{t_n\}\to\infty$, we will denote by $\rho^n$ the function $\rho^{(t_n)}$. The following lemma establishes compactness properties of the sequence $\{\rho^n\}$.
\begin{lemma} The following assertions hold:
\label{lem: compactness}
\begin{enumerate}[(1)]
\item The two sequences $\left\{Q_1^n\,,\,n\in\mathbb{N}\right\}$ and $\left\{Q_2^n\,,\,n\in\mathbb{N}\right\}$, defined in
\[
Q_1^n:=\int_{\X}\rho^n\log\rho^n \mathds{1}_{\rho^n\geq 1}\,d\x,
\]
and
\[
Q_2^n:=\int_{\X}\Big(\frac{|p|^2}{2}+\frac{\|z\|^2}{2}+V(q)+\frac{1}{2}U\ast \rho^n\Big)\rho^n\,d\x
\]
are uniformly bounded in time.
\item The sequence of distributions $\left\{\rho^n\,,\,n\in\mathbb{N}\right\}$ is relatively compact in $C([0;T],L^1(\X))$.
\end{enumerate}
\end{lemma}
\begin{proof}
The proof of this lemma follows the same lines as that of \cite[Lemma 3.5]{DuongTugaut2018} and~[Lemma 5.2] \cite{BonillaCarrilloSoler97}; hence we omit it here.
\end{proof}
\begin{lemma}
\label{lem: NL eqn} Let $\xi\in C^2(\R)$ with $\xi''\in L^\infty(\R)$ and $\xi(0)=0$. Then the following equation holds in the sense of distribution on $(0,T)\times \X$
\begin{equation}
\label{eq: NL eqn}
\frac{\partial \xi(\rho)}{\partial t}-\K[\rho]\xi(\rho)=\mathrm{Tr}(A)\Big(\rho\xi'(\rho)-\xi(\rho)\Big)-\beta^{-1}\xi''(\rho)A\nabla_z\rho\cdot\nabla_z\rho,
\end{equation}
where the operator $\K$ is defined in \eqref{eq:stationary oper}.
\end{lemma}
\begin{proof}
The proof follows the lines of \cite[Lemma 3.9]{BonillaCarrilloSoler97}. We sketch the main idea here. Multiplying the equation
$
\frac{\partial\rho}{\partial t}\rho-\K[\rho]\rho=0
$
with $\xi'(\rho)$, we obtain
\begin{equation}
\label{eq: temp1}
\frac{\partial \xi(\rho)}{\partial t}-\Big(\K[\rho]\rho\Big)\xi'(\rho)=0.
\end{equation}
Direct computations and using the following identities
\[
\div_z(A z\rho)\xi'(\rho)=\div_{z}(Az\xi(\rho))-\mathrm{Tr}(A)(\xi(\rho)-\rho\xi'(\rho)),
\]
and
\[
\div_z(A\nabla_z\rho)\xi'(\rho)=\div_z(A\nabla_z\xi(\rho))-\xi''(\rho)A\nabla_z\rho\cdot\nabla_z\rho,
\]
yield: 
\begin{equation}
\label{eq: temp2}
\Big(\K[\rho]\rho\Big)\xi'(\rho)=\K[\rho]\xi(\rho)-\mathrm{Tr}(A)(\xi(\rho)-\rho\xi'(\rho))-\xi''(\rho)A\nabla_z\rho\cdot\nabla_z\rho.
\end{equation}
Substituting \eqref{eq: temp2} into \eqref{eq: temp1} we obtain \eqref{eq: NL eqn}.
\end{proof}
\begin{proposition} There exists a function $\rho^\infty\in C([0,T];L^1(\X))$ such that
\[
\lim\limits_{n\to+\infty}\rho^n=\rho^\infty \quad\text{in}\quad C([0,T],L^1(\X)),
\]
and the following assertions hold
\begin{enumerate}[(1)]
\item $\sqrt{\rho^\infty}$ solves the following equation
\begin{equation}
\label{eq: eqn for sqrt of rho-infty}
\frac{\partial\sqrt{\rho^\infty}}{\partial t}-\K[\rho^\infty]\sqrt{\rho^\infty}=-\frac{\rm{Tr}(A)}{2}\sqrt{\rho^\infty}+\frac{\beta }{4}Az\cdot z\sqrt{\rho^\infty}.
\end{equation}
\item $\rho^\infty$ is a stationary solution to Equation~\eqref{eq: GLMV} and is of the form
\begin{equation}
\rho^\infty(q,p,z)=\frac{\exp\Big[-\beta\Big(\frac{\|z\|^2}{2}+\frac{|p|^2}{2}+V(q)+U\ast h^\infty(q)\Big)\Big]}{\int_{\X}\exp\Big[-\beta\Big(\frac{\|z\|^2}{2}+\frac{|p|^2}{2}+V(q)+U\ast h^\infty(q)\Big)\Big]\,d\x},
\end{equation}
where $h^\infty$ satisfies
\[
h^\infty=\frac{\exp\Big[-\beta\Big(V(q)+U\ast h^\infty(q)\Big)\Big]}{\int_{\R^d}\exp\Big[-\beta\Big(V(q)+U\ast h^\infty(q)\Big)\Big]\,dq},
\]
\end{enumerate}
\end{proposition}
\begin{proof}
By Lemma~\ref{lem: compactness} there exists a subsequence of $\{t_n\}$, which we denote with the same index, and a function $\rho^\infty\in C([0,T];L^1(\X))$ such that
\[
\rho^n\xrightarrow{n\to\infty}\rho^\infty \quad\text{in}\quad C([0,T],L^1(\X)).
\]
As a consequence
\[
\sqrt{\rho^n}\xrightarrow{n\to\infty}\sqrt{\rho^\infty} \quad\text{in}\quad L^2([0,T]\times \X)).
\]
We now derive the limiting equation for $\sqrt{\rho^\infty}$. To focus on the main idea, we provide a formal derivation here. Applying Lemma~\ref{lem: NL eqn} to $\xi(\rho^n)=\sqrt{\rho^n}$ we obtain 
\begin{equation}
\label{eq: eqn for sqrt of rho-n}
\frac{\partial \sqrt{\rho^n}}{\partial t}-\K[\rho^n]\sqrt{\rho^n}=-\frac{\rm{Tr}(A)}{2}\sqrt{\rho^n}+\frac{\beta^{-1}}{4}\frac{A\nabla_z\rho^n\cdot\nabla_z\rho^n}{(\rho^n)^{3/2}},
\end{equation}
in the sense of distribution. We wish to pass to the limit $n\to\infty$ in this equation. For which we need appropriate compactness properties. Equation~\eqref{eq: limit of derivative free energy} implies that
\[
\lim_{t\to\infty}\Big\|\frac{\beta^2}{4} Az\cdot z\rho^n -A\nabla_z\sqrt{\rho^n}\cdot\nabla_z\sqrt{\rho^n}\Big\|_{L^1([0,T]\times\X)}=0.
\]
In addition, similarly as in \cite{Tugaut13a,Tugaut13b}, we can prove that $\{\nabla F\ast\rho^n(q)\}$ converges uniformly on each compact set towards $\nabla F\ast\rho^\infty(q)$. We now pass to the limit $n\to\infty$ in Equation~\eqref{eq: eqn for sqrt of rho-n} to get
\begin{equation*}
\frac{\partial\sqrt{\rho^\infty}}{\partial t}-\K[\rho^\infty]\sqrt{\rho^\infty}=-\frac{\rm{Tr}(A)}{2}\sqrt{\rho^\infty}+\frac{\beta }{4}Az\cdot z\sqrt{\rho^\infty},
\end{equation*}
which proves \eqref{eq: eqn for sqrt of rho-infty}. This derivation can be made rigorous following the lines of \cite[Proof of Theorem 1.2]{BonillaCarrilloSoler97} by applying Lemma~\ref{lem: NL eqn} to $\xi_\vep(\rho^n)=\sqrt{\rho^n+\vep}-\sqrt{\vep}$  instead of the square root function; then passing first to $n\to\infty$ and then to $\vep\to 0$. We omit the details here and proceed to the second part of the proposition. From \eqref{eq: limit of derivative free energy}, we get that
\begin{equation}
\label{eq: limiting eqn for rho}
\frac{\beta}{2}	z\sqrt{\rho^\infty}+\nabla_z\sqrt{\rho^\infty}=0,
\end{equation}
in the sense of distributions on $[0,T]\times \X$. Multiplying \eqref{eq: limiting eqn for rho} with $\exp\Big(\frac{\beta \|z\|^2}{4}\Big)$ we obtain
\[
\nabla_z\bigg[\exp\Big(\frac{\beta \|z\|^2}{4}\Big)\sqrt{\rho^\infty}\bigg]=0,
\]
in the sense of distribution on $[0,T]\times \X$, which implies that there exists a function $g^\infty(t,q,p)\in L^1_{\rm loc}([0,T]\times\R^d\times\R^d)$ such that
\begin{equation}
\label{eq: rho-infty}
\sqrt{\rho^\infty}=\sqrt{g^\infty(t,q,p)}\exp\Big(-\frac{\beta \|z\|^2}{4}\Big),~~\text{i.e.,}~~\rho^\infty(t,q,p,z)=g^\infty(t,q,p)\exp\Big(-\frac{\beta\|z\|^2}{2}\Big).
\end{equation} 
Substituting this representation to Equation~\eqref{eq: eqn for sqrt of rho-infty} we obtain the following equation for $\sqrt{g^\infty}$
\begin{equation}
\label{eq: eqn for g-infty}
\frac{\partial\sqrt{g^\infty}}{\partial t}=-\div_q\Big(p \sqrt{g^\infty}\Big)+\div_p\Big((\nabla V(q)+\nabla_q U\ast g^\infty(q))\sqrt{g^\infty}\Big)-\lambda^Tz\cdot\Big(\nabla_p\sqrt{g^\infty}+\frac{\beta}{2}p\sqrt{g^\infty}\Big).
\end{equation}
Since $g^\infty$ is independent of $z$, we must have
\[
\frac{\beta}{2}	p\sqrt{g^\infty}+\nabla_p\sqrt{g^\infty}=0.
\]
This implies that there is a function $h^\infty(t,q)\in L^1_{\rm{loc}}([0,T]\times\R^d)$ such that
\begin{equation}
\label{eq: g infty}
g^\infty(t,q,p)=h^\infty(t,q)\exp\Big(-\frac{\beta|p|^2}{2}\Big).
\end{equation}
Substituting this back into \eqref{eq: eqn for g-infty} we get
\begin{equation*}
\frac{\partial\sqrt{h^\infty}}{\partial t}=-p\cdot\Big(\nabla_q\sqrt{h^\infty}+\frac{\beta}{2}(\nabla V(q)+\nabla U\ast h^\infty(q))\sqrt{h^\infty}\Big).
\end{equation*}
Since $h^\infty$ is independent of $p$, we must have that 
\[
\frac{\partial\sqrt{h^\infty}}{\partial t}=
\nabla_q\sqrt{h^\infty}+\frac{\beta}{2}(\nabla V(q)+\nabla U\ast h^\infty(q))\sqrt{h^\infty}=0,
\]
which implies that $h^\infty$ solves
\begin{equation}
\label{eq: h infty}
h^\infty=\frac{\exp\Big[-\beta(V(q)+U\ast h^\infty(q))\Big]}{\int_{\R^d}\exp\Big[-\beta(V(q)+U\ast h^\infty(q))\Big]\,dq}.
\end{equation}
Substituting \eqref{eq: h infty} and \eqref{eq: g infty} back to \eqref{eq: rho-infty} we obtain that
\[
\rho^\infty(q,p,z)=\frac{\exp\Big[-\beta\Big(\frac{\|z\|^2}{2}+\frac{|p|^2}{2}+V(q)+U\ast h^\infty(q)\Big)\Big]}{\int_{\X}\exp\Big[-\beta\Big(\frac{\|z\|^2}{2}+\frac{|p|^2}{2}+V(q)+U\ast h^\infty(q)\Big)\Big]\,d\x},
\]
where $h^\infty$ solves \eqref{eq: h infty}. According to Proposition~\ref{prop: characterization}, $\rho^\infty$ is a stationary solution to Equation~\eqref{eq: GLMV}. 
\end{proof}
%
%
%
%%%%%%%%%%%%%%%%%%%%%%%%%%%%%%%%%%%%%%%%%%%%%%%%%%%%
%
%
\section{GENERIC formulation of the GLMV}
\label{sec:GENERIC}

In this section, we recast the overdamped, underdamped and generalized McKean-Vlasov equations into the GENERIC (General Equation for Non-Equilibrium Reversible-Irreversible Coupling) formalism, thus putting them in a common framework. In addition, using the GENERIC formulation, we can rederive formulas for their stationary states obtained in Section \ref{sec: nonconvex}. 
\subsection{The GENERIC framework}
As suggested by its name, in the GENERIC framework an evolution equation for an unknown $\az$ in a state space $\aZ$ is decomposed into the sum of a reversible dynamics and an irreversible dynamics as follows
\begin{equation}
\label{eq:GENERICeqn1}
\frac{\partial\az}{\partial t}=\underbrace{\aL (\az)\,\frac{\delta \aE(\az)}{\delta \az}}_{\text{reversible dynamics}}+\underbrace{\aM(\az)\,\frac{\delta \aS(\az)}{\delta\az}}_{\text{irreversible dynamics}}.
\end{equation}
In the above equation, $\aL, \aM$ are two operators, $\aE, \aS:\aZ\rightarrow \R $ are two functionals that are respectively callled energy and entropy functionals, $\frac{\delta\aE}{\delta\az},\frac{\delta\aS}{\delta\az}$ are appropriate derivatives of $\aE$ and $\aS$ (such as either the Fr\'echet derivative or a gradient with respect to some inner product). The GENERIC framework impose the following conditions on $\{\aL,\aM,\aE,\aS\}$, which are called the GENERIC building blocks,
\begin{enumerate}[(C1)]
\item $\aL= \aL(\az)$ is, for each $\az$, antisymmetric operator satisfying the Jacobi identity
\begin{equation}
\label{def:Jacobi}
\{\{\aF_1,\aF_2\}_{\aL},\aF_3\}_{\aL}+\{\{\aF_2,\aF_3\}_{\aL},\aF_1\}_{\aL}+\{\{\aF_3,\aF_1\}_{\aL},\aF_2\}_{\aL}=0,
\end{equation}
for all functions $\aF_i\colon\aZ\rightarrow\R,\ i=1,2,3$,  where the Poisson bracket $\{\cdot,\cdot\}_{\aL}$ is defined via
\begin{equation}
\label{def:bracket}
\{\aF,\aG\}_{\aL}\colonequals\frac{\delta \aF}{\delta\az} \cdot \aL\,\frac{\delta\aG}{\delta\az}.
\end{equation}
\item $\aM=\aM(\az)$ is symmetric and positive semidefinite.
\item The following \textit{degeneracy condition} holds 
\begin{equation}
\label{def:degeneracy}
\aL\,\frac{\delta\aS}{\delta\az}=0,\quad
\aM\,\frac{\delta\aE}{\delta\az}=0.
\end{equation}
\end{enumerate}
Being formulated in the GENERIC framework, an evolution equation automatically justifies the first and second laws of thermodynamics, i.e., along its solutions energy is conserved while entropy is non-decreasing.  Indeed, they are direct consequences of the above conditions:
\begin{align*}
\frac{d\aE(\az(t))}{dt} & =\frac{\delta\aE}{\delta\az}\cdot \frac{d\az}{dt}=\frac{\delta\aE}{\delta\az}\cdot\left(\aL\,\frac{\delta\aE}{\delta\az}+\aM\,\frac{\delta\aS}{\delta\az}\right)= \frac{\delta\aE}{\delta\az}\cdot \aL\,\frac{\delta\aE}{\delta\az}+\frac{\delta\aE}{\delta\az}\cdot \aM\,\frac{\delta\aS}{\delta\az}= 0,\\
\frac{d\aS(\az(t))}{dt} & =\frac{\delta\aS}{\delta\az}\cdot \frac{d\az}{dt}=\frac{\delta\aS}{\delta\az}\cdot\left(\aL\,\frac{\delta\aE}{\delta\az}+\aM\,\frac{\delta\aS}{\delta\az}\right)=\frac{\delta\aS}{\delta\az}\cdot\aM\,\frac{\delta\aS}{\delta\az}\geq0.
\end{align*}
In addition, in the GENERIC framework, equilibria can be obtained by the maximum entropy principle: they are the maximizers of the entropy $\aS$ under the constraint that the energy is constant $\aE(\az)=\aE_0$ \cite{OG97a,Mielke11}. Define $\Phi(\az)=\aS(\az)-\lambda (\aE(\az)-\aE_0)$ for some $\lambda\in \R$, which plays the role of a Lagrange multiplier. The states $\az_\infty$ that maximize the entropy under the constraint $\aE(\az)=E_0$ are solutions to
\begin{equation*}
\begin{cases}
\frac{\delta \Phi(\az)}{\delta \az}=0,\\
\aE(\az)=\aE_0,
\end{cases}
\end{equation*}
which is equivalent to
\begin{equation}
\label{eq: stationary}
\begin{cases}
\frac{\delta \aS(\az)}{\delta \az}=\lambda\frac{\delta \aE(\az)}{\delta \az},\\
\aE(\az)=\aE_0.
\end{cases}
\end{equation}
Furthermore, since $\Phi(z_t)$ is also a nondecreasing function of time, it follows that all solutions $\az_t$ of the GENERIC equation \eqref{eq:GENERICeqn1} converge to $\az_\infty$ as $t\rightarrow \infty$. We refer to \cite[Property 3]{OG97a} for more detailed discussion.

In the next sections, we recast the MV dynamics, the VFP dynamics and the GLMV dynamics into the GENERIC framework and we use Equation~\eqref{eq: stationary} to establish formulas for their stationary solutions recovering the results in the previous section.
\subsection{GENERIC formulation of the generalized McKean-Vlasov equation}
It has been known that both the overdamped and underdamped McKean-Vlasov equations are Wasserstein gradient flows~\cite{carrillo2003,Carrillo2006}. More recently, in~\cite{DPZ13} it was shown that the underdamped McKean-Vlasov equation can be cast into the GENERIC framework. In this section, we apply the technique therein for the generalized McKean-Vlasov equation. To this end, we need to construct the building block $\{\aZ,\aL,\aM,\aE,\aS\}$ and verify Conditions $(C1)-(C3)$ in the previous section.

Let $\rho$ be a smooth solution of the the generalized McKean-Vlasov equation \eqref{eq: GLMV}. We define
\begin{equation}
\Hc(\rho):=\int\Big(\frac{|p|^2}{2}+V(q)+\frac{1}{2}(U\ast\rho)+\frac{\|z\|^2}{2}\Big)\rho\,d\x. 
\end{equation}
Let us introduce an auxiliary time-dependent variable $e$ whose evolution is given by
\begin{equation}
\label{eq: eqn e}
\frac{d}{dt}e=-\frac{d}{dt}\Hc(\rho_t)=\int A z\cdot z\rho\,d\x-\beta^{-1}\rm{Tr}(A).
\end{equation}
We now show that the coupled system for the variables $(\rho,e)$,
\begin{equation}
\label{eq: coupled system}
\partial_t\begin{pmatrix}
\rho\\e
\end{pmatrix}=\begin{pmatrix}
\div(\rho\J\nabla \H_\rho)+\beta^{-1}\div(\rho A\nabla_z\log\rho)+\div(\rho A z)\\
-\beta^{-1}\rm{Tr}(A)+\int Az\cdot z\rho\,d\x
\end{pmatrix},
\end{equation}
can be formulated in the GENERIC framework. We construct the building blocks as follows
\begin{align*}
&\az=(\rho,e),\quad \aZ=\mathcal{P}_2(\X)\times \R,\quad \aE(\rho,e)=\Hc(\rho)+e,\quad \aS(\rho,e)=-\beta^{-1}\int\rho\log\rho+e,
\\&\aL(\rho,e)=\begin{pmatrix}
\aL_{\rho\rho}&0
\\ 0& 0
\end{pmatrix},\quad \text{and}\quad \aM(\rho,e)=\begin{pmatrix}
\aM_{\rho\rho}&\aM_{\rho e}\\
\aM_{e \rho}& \aM_{e e},
\end{pmatrix}
\end{align*}
where the operators defining $\aL$ and $\aM$ are given by
\begin{align*}
&\aL_{\rho\rho}\xi=\div\Big(\rho \J\nabla\xi\Big),\quad \aM_{\rho\rho}\xi=-\div(\rho A\nabla_z\xi),\quad \aM_{\rho e}r=r\div\Big(\rho Az\Big)
\\& \aM_{e\rho}\xi=-\int Az\cdot\nabla_z\xi\, \rho\,d\x,\quad \aM_{e e}r=r\int Az\cdot z\,\rho\,d\x
\end{align*}
with $\J$ being the antisymmetric matrix defined in~\eqref{eq: H and J}. We compute the first variations of $\aE$ and $\aS$:
\begin{equation}
\label{eq: derivative}
\frac{\delta\aE}{\delta\az}=\begin{pmatrix}
V(q)+\frac{|p|^2}{2}+U\ast\rho +\frac{\|z\|^2}{2}\\1
\end{pmatrix}=\begin{pmatrix}
\H_\rho\\1
\end{pmatrix},\qquad\frac{\delta\aS}{\delta\az}=\begin{pmatrix}
-\beta^{-1}(\log\rho+1)\\1
\end{pmatrix}.
\end{equation}
Thus the GENERIC equation associated to these blocks is given by
\begin{align*}
\partial_t\begin{pmatrix}
\rho\\e
\end{pmatrix}=\aL\frac{\delta\aE}{\delta\az}+\aM\frac{\delta\aS}{\delta\az}&=\begin{pmatrix}
\aL_{\rho\rho}\H_\rho\\0
\end{pmatrix}+\begin{pmatrix}
-\beta^{-1}\aM_{\rho\rho}(\log\rho+1)+\aM_{\rho e} 1\\
-\beta^{-1}\aM_{e \rho}(\log\rho+1)+\aM_{ee} 1
\end{pmatrix}
\\&=\begin{pmatrix}
\div(\rho\J\nabla \H_\rho)\\0
\end{pmatrix}+\begin{pmatrix}
\beta^{-1}\div(\rho A\nabla_z\log\rho)+\div(\rho A z))\\\beta^{-1}\int \rho\,A z\cdot\nabla_z\log \rho\,d\x+\int Az\cdot z\rho\,d\x
\end{pmatrix}
\\&=\begin{pmatrix}
\div(\rho\J\nabla \H_\rho)+\beta^{-1}\div(\rho A\nabla_z\log\rho)+\div(\rho A z)\\
\beta^{-1}\int \rho A z\cdot\nabla_z\log \rho\,d\x+\int Az\cdot z\rho\,d\x
\end{pmatrix},
\end{align*}
which is indeed the system~\eqref{eq: coupled system}. By straight but lengthy computations one can verify all the conditions imposed on the GENERIC framework for the building blocks constructed above. 
\subsection{Invariant measures}
We now look for stationary solutions of the coupled system~\eqref{eq: coupled system} using~\eqref{eq: stationary}. A stationary solution $\az_\infty=(\rho_\infty,e_\infty)$, where $\rho_\infty$ is a probability measure, of \eqref{eq: GLMV}-\eqref{eq: eqn e} maximizes the entropy $\aS(\az)$ under the constraints that $\aE(\az)$ is a constant and that $\int \rho (dqdp)=1$. Define $\Phi(\az)=\aS(\az)-\lambda_1(\aE(\az)-\aE_0)-\lambda_2 (\int_{\R^{2d}} \rho (dqdp)-1)$ for some $\lambda_1,\lambda_2\in \R$ then $(\rho_\infty,e_\infty)$ satisfies the following equation
\begin{equation}
\label{eq: stationary3}
\begin{cases}
\frac{\delta \Phi(\az)}{\delta\az}=0,\\
\aE(\az)=\aE_0,\\
\int \rho_\infty(dqdpdz)=1.
\end{cases}
\end{equation}
Using the computations in \eqref{eq: derivative}, we have
$$
\frac{\delta \Phi(\az)}{\delta\az}=\frac{\delta \aS(\az)}{\delta\az}-\lambda_1\frac{\delta \aE(\az)}{\delta\az}-\begin{pmatrix}
\lambda_2\\0
\end{pmatrix}=\begin{pmatrix}
-\beta^{-1}(\log\rho+1)-\lambda_1 \H_\rho-\lambda_2\\
1-\lambda_1
\end{pmatrix}.
$$
Therefore, equation \eqref{eq: stationary3} becomes
\begin{equation*}
\begin{cases}
-\beta^{-1} (\log \rho_\infty +1)=\lambda_1 \H_{\rho_\infty}+\lambda_2,\\
1=\lambda_1,\\
\H_\rho+e=\aE_0,\\
\int \rho_\infty(dqdp)=1.
\end{cases}
\end{equation*}
Substituting $\lambda_1=1$ into the first equation and then combining with the last equation, we obtain that $\rho_\infty$ satisfies
$$
\rho_\infty=\frac{\exp\big(-\beta \H_{\rho_\infty}\big)}{\int \exp\big(-\beta \H_{\rho_\infty}\big)\,d\x}=\frac{\exp\Big(-\beta(V(q)+\frac{|p|^2}{2}+U\ast\rho +\frac{\|z\|^2}{2})\Big)}{\int \exp\Big(-\beta(V(q)+\frac{|p|^2}{2}+U\ast\rho +\frac{\|z\|^2}{2})\Big)}.
$$
This is exactly Equation \eqref{eq: invariant} obtained in Section \ref{sec: nonconvex}. Similarly we can derive the formula in Corollary \ref{cor: stationary} for a stationary solution $\rho_\infty^{VFP}$ of the VFP dynamics using its GENERIC formulation in \cite{DPZ13}.
\section{White noise (Markovian) limits}
\label{sec: asymptotic limit}
In this section, we derive the underdamped McKean-Vlasov dynamics from the white noise (Markovian) limit of the generalized McKean-Vlasov dynamics both for the particle system and for the mean-field PDE. We apply the formal perturbation expansions method developed in \cite{PavliotisStuart2008,Pavl2014}. 
\subsection{From the generalized to the underdamped Langevin equations}
We will derive \eqref{e: VFP particles} from \eqref{eq: GLMV particles} in the limit of vanishing correlation time of the noise which corresponds to rescaling $\lambda$ and $A$ in \eqref{eq: GLMV particles} according to $\lambda\mapsto\lambda/\vep$ and $A\mapsto A/\vep^2$. The SDE~\eqref{eq: GLMV particles} becomes
\begin{subequations}
\label{eq: GLMV particles 2}
\begin{eqnarray}
dQ_i^\vep&=&P_i^\vep\,dt,
\\dP_i^\vep&=&-\nabla V(Q_i^\vep)\,dt-\frac{1}{N}\sum_{j=1}^N \nabla U(Q_i^\vep-Q_j^\vep)\,dt+\frac{1}{\vep}\lambda^T Z_i^\vep \,dt,
\\ dZ_i^\vep&=&-\frac{1}{\vep} P_i^\vep \lambda \,dt-\frac{1}{\vep^2} A Z_i\,dt+\frac{1}{\vep}\sqrt{2 \beta^{-1} A}dW_i.
\end{eqnarray}
\end{subequations}
\begin{proposition}
\label{prop: limits particle}
Let $\{Q_i^\vep, P_i^\vep, Z_i^\vep\}$ be the solution of \eqref{eq: GLMV particles 2}, and assume that the matrix $A$ is invertible. Then $(Q_i^\vep (t), P_i^\vep (t))$ converges weakly to the solution of the Vlasov-Fokker-Planck equation \eqref{e: VFP particles}, where the friction coefficient is given by the formula
\begin{equation}
\label{eq: friction formula}
\gamma=\langle \lambda, A^{-1}\lambda\rangle.
\end{equation}
\end{proposition}
\begin{proof}
We rewrite \eqref{eq: GLMV particles 2} in a compact matrix form 
\begin{subequations}
\label{eq: GLMV particles 3}
\begin{eqnarray}
d\Qu^\vep&=&\Pu^\vep\,dt,
\\ d\Pu^\vep&=& F(\Qu)\,dt+\frac{1}{\vep}\lambda^T \Zu^\vep \,dt,
\\ d\Zu^\vep&=&-\frac{1}{\vep} \Pu^\vep \lambda \,dt-\frac{1}{\vep^2} A \Zu^\vep \,dt+\frac{1}{\vep}\sqrt{2 \beta^{-1} A}\, d\Wu.
\end{eqnarray}
\end{subequations}
where we have used the following notation
\[
\Qu^\vep=(Q_1^\vep,\ldots,Q_N^\vep)^T,\quad  \Pu^\vep=(P_1^\vep,\ldots,P_N^\vep)^T,\quad 
\Zu^\vep=(Z_1^\vep,\ldots,Z_N^\vep)^T,\quad \Wu=(W_1,\ldots,W_N)^N,
\]
and
\[
F(\Qu^\vep)=\Big(-\nabla V(Q_1^\vep)-\frac{1}{N}\sum_{j=1}^N \nabla U(Q_1^\vep-Q_j^\vep), \ldots, -\nabla V(Q_N^\vep)-\frac{1}{N}\sum_{j=1}^N \nabla U(Q_N^\vep-Q_j^\vep)\Big)^T.
\]
The backward Kolmogorov equation associated to \eqref{eq: GLMV particles 3} is
\begin{equation}
\label{eq: GLMV exp}
\frac{\partial u^\vep}{\partial t}=\Big(\frac{1}{\vep^2}\L_0+\frac{1}{\vep}\L_1+\L_2\Big) u,
\end{equation}
with
\begin{eqnarray*}
\L_0&=& - A \zu \cdot \nabla_{\zu}+\beta^{-1} A:D^2_{\zu},
\\ \L_1&=& \lambda^T \zu\cdot\nabla_{\pu}-\pu \lambda\cdot \nabla_{\zu},
\\ \L_2&=&\pu\cdot\nabla_{\qu}+F(\qu)\cdot\nabla_{\pu}.
\end{eqnarray*}
We seek for a solution to \eqref{eq: GLMV exp} in the form of a power series expansion in $\vep$:
$$
u^\vep=u_0+\vep u_1+\vep^2 u_2+\dots.
$$
By substituting this expansion into \eqref{eq: GLMV exp} and equating powers of $\vep$, we obtain the following sequence of equations
\begin{subequations}
\begin{eqnarray}
\L_0 u_0&=&0,\label{e: eqn1}
\\ -\L_0 u_1&=&\L_1 u_0,\label{e: eqn2}
\\-\L_0 u_2&=& \L_1 u_1+\L_2 u_0-\frac{\partial u_0}{\partial t}, \label{e: eqn3}
\\ \ldots &=&\ldots.\notag
\end{eqnarray}
\end{subequations}
It follows from the first equation \eqref{e: eqn1} that to leading order, the solution of the Kolmogorov equation is independent of the auxiliary variables $\zu$, $u_0 = u(\qu, \pu, t)$. The solvability of the second equation \eqref{e: eqn2} reads
$$
\int_{\R^{dmN}}\L_1 u_0 \, e^{-\frac{\beta}{2}\|\zu\|^2}\,d\zu=0,
$$
which is fulfilled since $\L_1 u_0=\lambda^T\zu \cdot\nabla_{\pu}u$, thus
$$
\int_{\R^{dmN}}\L_1 u_0 e^{-\frac{\beta}{2}\|\zu\|^2}\,d\zu= \int_{\R^{dmN}}\big(\lambda^T\zu \cdot\nabla_{\pu}u\big)\, e^{-\frac{\beta}{2}\|\zu\|^2}\,d\zu=0.
$$
Equation~\eqref{e: eqn2} becomes
$$
-\L_0 u_1=\lambda^T\zu \cdot\nabla_{\pu}u
$$
The solution to this equation is
$$
u_1(q,p,t)=\zu A^{-1}\lambda\cdot\nabla_{\pu} u.
$$
Note that in the above equation, we ignore the part of the solution that lies in the null space of $\L_0$ since it will not affect the limiting equation. The solvability condition for the third equation \eqref{e: eqn3} gives
$$
\int \Big(\L_1 u_1+\L_2 u-\frac{\partial u}{\partial t}\Big) (2\pi \beta^{-1})^{dmN} e^{-\frac{\beta}{2}\|\zu\|^2}\,d\zu=0.
$$
Since $\L_2 u-\frac{\partial u}{\partial t}$ is independent of $\zu$, we deduce\,  from the above solvability equation that
\begin{equation}
\label{eq: u}
\frac{\partial u}{\partial t}=\L_2 u +\int (\L_1 u_1) (2\pi \beta^{-1})^{dmN} e^{-\frac{\beta}{2}\|\zu\|^2}\,d\zu.
\end{equation}
We compute the second term on the right-hand side of \eqref{eq: u} using the properties of Gaussian distributions:
$$
\int (\L_1 u_1) (2\pi \beta^{-1})^{dmN} e^{-\frac{\beta}{2}\|\zu\|^2}\,d\zu=\langle \lambda, A^{-1}\lambda\rangle \Big(-\pu\cdot\nabla_{\pu} u+\beta^{-1} \Delta_{\pu} u\Big).
$$
Substituting this calculation back into \eqref{eq: u} we obtain that $u$ satisfies the following PDE
$$
\frac{\partial u}{\partial t}=\pu\cdot\nabla_{\qu}u +F(\qu)\cdot\nabla_{\pu} u+\langle \lambda, A^{-1}\lambda\rangle \Big(-\pu\cdot\nabla_{\pu} u+\beta^{-1} \Delta_{\pu} u\Big).
$$
This is exactly the backward Kolmogorov equation of
the Vlasov-Fokker-Planck system \eqref{e: VFP particles} with $\gamma=\langle \lambda, A^{-1}\lambda\rangle$. 
\end{proof}
\subsection{From the generalized McKean-Vlasov equation to the underdamped Mckean-Vlasov equation at the mean-field level}
Under the same rescaling, $\lambda\mapsto \lambda/\vep$ and $A\mapsto A/\vep^2$, as in the previous section system \eqref{e:mckean} becomes
\begin{subequations}\label{eq: scaledGLMV}
\begin{eqnarray}
dQ^\vep(t)&=&P^\vep(t)\,dt,
\\dP^\vep(t)&=&-\nabla V(Q^\vep(t))\,dt-\nabla_q U\ast\rho^\vep_t(Q^\vep_t)\,dt+\frac{1}{\vep}\lambda^T Z^\vep(t)\,dt,
\\ dZ^\vep(t)&=&-\frac{1}{\vep}	P^\vep(t) \lambda\,dt-\frac{1}{\vep^2} AZ^\vep(t)+\frac{1}{\vep}\sqrt{2\beta^{-1}A}dW(t).
\end{eqnarray}
\end{subequations}
\begin{proposition}
\label{prop: limits mean-field}
Let $(Q^\vep, P^\vep, Z^\vep)$ be the solution of \eqref{eq: scaledGLMV}, and assume that the matrix $A$ is invertible. Then $(Q^\vep, P^\vep)$ converges weakly to the solution of the Vlasov-Fokker-Planck equation
\begin{subequations}
\label{eq: VFP SDE}
\begin{eqnarray}
dQ(t)&=&P(t)\,dt,
\\dP(t)&=&-\nabla V(Q(t))\,dt-\nabla_q U\ast\rho_t(Q(t))\,dt-\gamma P(t)\,dt +\sqrt{2\gamma\beta^{-1}} dW(t),
\end{eqnarray}
\end{subequations}
where the friction coefficient is given by formula \eqref{eq: friction formula}.
\end{proposition}

The Fokker-Planck equation associated to the McKean SDE~\eqref{eq: scaledGLMV} is
\begin{eqnarray}
\label{eq: FPscaledGLMV}
\frac{\partial \rho^\vep }{\partial t}&=&\L^\ast\rho^\vep\notag
\\&=& -p\cdot \nabla_q \rho^\vep+(\nabla V(q)+\nabla_q U(q)\ast\rho^\vep_t)\cdot\nabla_p\rho^\vep+\frac{1}{\vep}\Big(-\lambda^T z\cdot\nabla_p\rho^\vep+p\lambda\cdot\nabla_z \rho^\vep\Big)\notag
\\&&\qquad +~\frac{1}{\vep^2}\Big(\div_p(Az\rho^\vep)+\beta^{-1}\div(A\nabla_p\rho^\vep)\Big)\notag
\\&:=&\bigg(\L_2^\ast+\frac{1}{\vep}\L_1^\ast+\frac{1}{\vep^2}\L_0^\ast\bigg)\rho^\vep.
\end{eqnarray}
According to Proposition~\ref{prop: characterization}, an invariant measure $\rho_\infty$ of the equation~\eqref{eq: scaledGLMV}, if it
exists, satisfies
\begin{equation}
\label{eq: invariant 2}
\rho_\infty(q,p,z)=\frac{\exp\Big[-\beta\,\Big(\frac{|p|^2}{2}+\frac{\|z\|^2}{2}+V(q)+U\ast\rho_\infty(q)\Big) \Big]}{\int \exp\Big[-\beta\,\Big(\frac{|p|^2}{2}+\frac{\|z\|^2}{2}+V(q)+U\ast\rho_\infty(q)\Big) \Big]\,dqdpdz}.
\end{equation}
We define the function $f^\vep(q,p,z,t)$ through
\begin{equation}
\label{eq: f}
\rho^\vep(q,p,z,t)=\rho_\infty(q,p,z)f^\vep(q,p,z,t).
\end{equation}
\begin{lemma} 
\label{lem: eqn f}
The function $f^\vep(q,p,z,t)$ defined in \eqref{eq: f} satisfies the
equation
\begin{align}
\label{eq: eqn f}
\frac{\partial f^\vep}{\partial t}&=-p\cdot\nabla_q f^\vep+(\nabla V(q)+\nabla_q U(q)\ast (f^\vep\rho_\infty))\cdot\nabla_p f^\vep+\beta f^\vep p\cdot\nabla U(q)\ast \rho_\infty(1-f^\vep)\notag
\\&\quad\quad+\frac{1}{\vep}\Big(-\lambda^Tz \cdot\nabla_p f^\vep+p\lambda\cdot\nabla_z f^\vep\Big)+\frac{1}{\vep^2}\Big(-Az\cdot\nabla_z f^\vep+\beta^{-1}\div_z(A\nabla_z f^\vep)\Big)\notag
\\&=:\Big(\hat{\L}_2+\frac{1}{\vep}\hat{\L}_1+\frac{1}{\vep^2}\hat{\L}_0\Big) f^\vep.
\end{align}
\end{lemma}
\begin{proof}
For simplicity of notation, we drop the superscript $\vep$ on $f^\vep$. From the definition of $f$ and $\rho_\infty$, we compute
\begin{equation}
\label{eq: eqn1}
\frac{\nabla_z\rho}{\rho}=\frac{\nabla_z(f\rho_\infty)}{ f\rho_\infty}=\frac{f\nabla_z\rho_\infty+\rho_\infty\nabla_z f}{f\rho_\infty}=\frac{-\beta f\rho_\infty z+\rho_\infty\nabla_z f}{f\rho_\infty}=-\beta z+ \frac{\nabla_z f}{f}.
\end{equation}
Therefore, we obtain
\begin{align}
\label{eq: 1st term}
\L_0^\ast\rho=\L_0^\ast(\rho_\infty f)&=\div_z\bigg[A\rho\Big(z+\beta^{-1}\frac{\nabla_z\rho}{\rho}\Big)\bigg]\notag
\\&\overset{\eqref{eq: eqn1}}{=}\beta^{-1}\div_z\Big(\rho_\infty A\nabla_z f\Big)\notag
\\&=\beta^{-1}\Big(\nabla_z\rho_\infty\cdot A\nabla_z f+\rho_\infty\div_z(A\nabla_z f)\Big)\notag
\\&=\rho_\infty\Big(-Az\cdot\nabla_z f+\beta^{-1}\div_z(A\nabla_z f)\Big).
\end{align}
%We note that \eqref{eq: 1st term} can also be derived from the general relationship between solutions of a Markov operator and its dual given in \eqref{eq: dual sols}. 
%
%Indeed, since $\L_0$ \textcolor{red}{is reversible with respect to $\rho_\infty$}, according to \eqref{eq: dual sols} we have
%$$
%\L_0^\ast(\rho_\infty f)=\rho_\infty \L_0 f=\rho_\infty\Big(-Az\cdot\nabla_z f+\beta^{-1}\div_z(A\nabla_z f)\Big).
%$$ 
We proceed with $\L_1^\ast\rho$:
\begin{align}
\L_1^\ast\rho=\L_1^\ast(f\rho_\infty)&=-\lambda^T z\cdot \nabla_p(f\rho_\infty)+p\lambda\cdot \nabla_z (f\rho_\infty)\notag
\\&=\rho_\infty\Big(-\lambda^Tz \cdot\nabla_p f+p\lambda\cdot\nabla_z f\Big)+f(-\lambda^Tz\cdot\nabla_p\rho_\infty + p\lambda\cdot\nabla_z \rho_\infty)\notag
\\&=\rho_\infty\Big(-\lambda^Tz \cdot\nabla_p f+p\lambda\cdot\nabla_z f\Big).
\label{eq: 2nd term}
\end{align}
Finally we compute $\L_2^\ast\rho$:
\begin{align}
\L_2^\ast\rho&=\L_2^\ast(f\rho_\infty)\notag
\\&=-p\cdot\nabla_q(f\rho_\infty)+(\nabla V(q)+\nabla_q U(q)\ast\rho_t)\cdot\nabla_p (f\rho_\infty)\notag
\\&=-p\cdot(f\nabla_q\rho_\infty+\rho_\infty \nabla_q f)+(\nabla V(q)+\nabla_q U(q)\ast\rho_t)\cdot (f\nabla_p\rho_\infty+\rho_\infty\nabla_p f)\notag
\\&=\Big[-p\cdot\big(-\beta f (\nabla V(q)+\nabla U(q)\ast\rho_\infty)+\nabla_q f\big)+(\nabla V(q)+\nabla_q U(q)\ast\rho_t)\cdot\big(-\beta f p+\nabla_p f \big)\Big]\, \rho_\infty\notag
\\&=\Big[-p\cdot\nabla_q f+(\nabla V(q)+\nabla_q U(q)\ast\rho_t)\cdot\nabla_p f+\beta f p\cdot\nabla U(q)\ast(\rho_\infty-\rho_t)\Big]\,\rho_\infty.
\label{eq: 3rd term}
\end{align}
Substituting \eqref{eq: 1st term},\eqref{eq: 2nd term} and \eqref{eq: 3rd term} into \eqref{eq: FPscaledGLMV} we obtain
\begin{align*}
\frac{\partial f}{\partial t}&=-p\cdot\nabla_q f+(\nabla V(q)+\nabla_q U(q)\ast (f\rho_\infty))\cdot\nabla_p f+\beta f p\cdot\nabla U(q)\ast \rho_\infty(1-f)\notag
\\&\qquad+\frac{1}{\vep}\Big(-\lambda^Tz \cdot\nabla_p f+p\lambda\cdot\nabla_z f\Big)+\frac{1}{\vep^2}\Big(-Az\cdot\nabla_z f+\beta^{-1}\div_z(A\nabla_z f)\Big),
\end{align*}
which is equation~\eqref{eq: eqn f}. This completes the proof of the lemma.
\end{proof}
We are now ready to prove Proposition~\ref{prop: limits mean-field}.
\begin{proof}[Proof of Proposition \ref{prop: limits mean-field} ]
We use the perturbation expansion method similarly as the proof of Proposition~\ref{prop: limits particle}. We look for a solution to \eqref{eq: eqn f} of the form
$$
f^\vep=f_0+\vep f_1+\vep^2 f_2+\dots.
$$ 
Substituting this expansion into \eqref{eq: eqn f} and equating powers of $\vep$ we obtain the following sequence of equations
\begin{subequations}
\begin{eqnarray}
\hat{\L}_0 f_0&=&0\label{eq: order-minus-2},
\\ -\hat{\L}_0 f_1&=&\hat{\L}_1 f_0 \label{eq: order-minus-1},
\\ -\hat{\L}_0 f_2&=&\hat{\L}_2 f_0+\hat{\L}_1 f_1-\frac{\partial f_0}{\partial t}\label{eq: order-0},
\\ \dots&=&\dots \notag
\end{eqnarray}
\end{subequations}
It follows from equation \eqref{eq: order-minus-2} that $f_0$ is independent of $z$:
$$
f_0= f(q,p,t).
$$
We compute the right-hand side of \eqref{eq: order-minus-1}:
\begin{align*}
\hat{\L}_1 f_0=-\lambda^T z\cdot\nabla_p f.
\end{align*}
Thus equation \eqref{eq: order-minus-1} becomes
$$
\hat{\L}_0 f_1=\lambda^T z\cdot\nabla_p f.
$$
This equation satisfies the solvability condition since $\lambda^T z\cdot\nabla_p f$ is orthogonal to the null space of $\hat{\L}_0^\ast$ which consists of functions of the form $e^{-\beta \frac{\|z\|^2}{2}} u(q,p)$. Therefore, it has a unique solution, up to a term in the null space of $\hat{\L}_0^\ast$,
$$
f_1=-z A^{-1}\lambda\cdot\nabla_p f.
$$
From this expression we can compute
\begin{equation}
\label{eq: L1f1}
\hat{\L}_1 f_1= \lambda^T A^{-1}\lambda \Big(\|z\|^2 \Delta_p f-p\cdot\nabla_p f\Big).
\end{equation}
The solvability condition for equation \eqref{eq: order-0} is that its right-hand side must be orthogonal to the null of $\hat{\L}_0^\ast$, i.e.,
\begin{equation}
\label{eq: sol cond}
\int \Big(\hat{\L}_2 f+\hat{\L}_1 f_1-\frac{\partial f}{\partial t}\Big) Z^{-1} e^{-\beta \frac{\|z\|^2}{2}}\,dz=0.
\end{equation}
Since $\hat{\L}_2 f-\frac{\partial f}{\partial t}$ does not depends on $z$, we conclude from \eqref{eq: sol cond} that
\begin{equation}
\label{eq: eqn for f0-1}
\frac{\partial f}{\partial t}=\hat{\L}_2 f+\int  (\hat{\L}_1 f_1) Z^{-1} e^{-\beta \frac{\|z\|^2}{2}}\,dz.
\end{equation}
Since $f$ does not depend on $z$ and due to the separability of variables in $\rho_\infty$, the first term in the right-hand of \eqref{eq: eqn for f0-1} can be simplified to
$$
\hat{\L}_2 f=-p\cdot\nabla_q f+(\nabla V(q)+\nabla_q U(q)\ast (f\hat{\rho}_\infty))\cdot\nabla_p f+\beta f p\cdot\nabla U(q)\ast \hat{\rho}_\infty(1-f),
$$
where $\hat{\rho}_\infty$ satisfies
\begin{equation}
\label{eq: stationary2}
\hat{\rho}_\infty(q,p)=\frac{\exp\Big[-\beta\,\Big(\frac{|p|^2}{2}+V(q)+U\ast\hat{\rho}_\infty(q)\Big) \Big]}{\int \exp\Big[-\beta\,\Big(\frac{|p|^2}{2}+V(q)+U\ast\hat{\rho}_\infty(q)\Big) \Big]\,dqdp}.
\end{equation}
The second term in the right-hand side of \eqref{eq: eqn for f0-1} can be computed using \eqref{eq: L1f1}
$$
\int  (\hat{\L}_1 f_1) Z^{-1} e^{-\beta \frac{\|z\|^2}{2}}\,dz=\lambda^T A^{-1}\lambda\Big(\beta^{-1}\Delta_p f-p\cdot\nabla_p f\Big).
$$
Substituting the above computation back to \eqref{eq: eqn for f0-1}, we obtain that
$$
\frac{\partial f}{\partial t}=-p\cdot\nabla_q f+(\nabla V(q)+\nabla_q U(q)\ast (f\hat\rho_\infty))\cdot\nabla_p f+\beta f p\cdot\nabla U(q)\ast \hat\rho_\infty(1-f)+\lambda^T A^{-1}\lambda\Big(\beta^{-1}\Delta_p f-p\cdot\nabla_p f\Big).
$$
We define $\hat{\rho}(q,p,t)$ through
$$
\hat{\rho}(q,p,t)=f(q,p,t)\hat\rho_\infty(q,p).
$$
Then analogously as in Lemma \ref{lem: eqn f}, we conclude that $\hat{\rho}(q,p,t)$ satisfies the Fokker-Planck equation associated to the Vlasov-Fokker-Planck dynamics \eqref{eq: VFP SDE} with the friction coefficient $\gamma$ given by $\gamma=\lambda^T A^{-1}\lambda$. We thus complete the proof of Proposition \ref{prop: limits mean-field}.
\end{proof}
%
%%%%%%%%%%%%%%%%%%%%%%%%%%%%%%%%%%%%%%%%%%%%%%%%%%%%%%%%
%
\section{Discussion and Future work}
\label{sec: discussion}
In this paper we studied the generalized McKean-Vlasov equation, that arises as the mean field limit of a system of interacting non-Markovian Langevin equations, under the assumption that the memory in the system can be described by introducing a finite number of auxiliary variables. We provided explicit formulas for the fundamental solution of the McKean-Vlasov equation for quadratic confining and interaction potentials, we studied the form of stationary states and we showed that no additional stationary states appear due to the memory in the system. Furthermore, we showed, under appropriate assumptions on the confining and interaction potentials, exponentially fast convergence to an equilibrium, we showed how the generalized McKean-Vlasov equation can be written in the GENERIC form and we studied the white noise limit.

As mentioned in the introduction section, the starting point of the present work is a system of weakly interacting diffusions in an extended phase space, where the finite number of auxiliary processes describe the memory of the system. Perhaps, a more natural starting point would be a system of weakly interacting particles that is coupled to one or more heat baths, e.g. the model studied in~\cite{EPR99}, but for weakly interacting nonlinear oscillators. Passing to the mean field limit and eliminating the heat bath variables are two operations that do not commute, in general. The rigorous study of the combined mean field and thermodynamic limits in, e.g. a Kac-Zwanzig type model~\cite{FKM65, GKS04} is an interesting and difficult problem that we will leave for further study. 

Other possible extensions of the present work include: first, the rigorous proof of a propagation of chaos result for non-Markovian interacting  Langevin equations in particular without the quasi-Markovianity assumption; second, the development of a complete existence and uniqueness of solutions theory for the generalized McKean-Vlasov equation; in addition, the study of non-Markovian interacting processes without the quasi-Markovianity assumption; and finally the study of the effect of coloured noise on the McKean-Vlasov dynamics, and in particular on the number and type of phase transitions.

\section* {\bf Acknowledgments}

GP is partially supported by the EPSRC under Grants No. EP/P031587/1, EP/L024926/1, and EP/L020564/1.

\bibliographystyle{alpha}
\bibliography{ref}

\newcommand{\etalchar}[1]{$^{#1}$}
\begin{thebibliography}{WLEC17}

\bibitem[BCS97]{BonillaCarrilloSoler97}
L.~L. Bonilla, J.~A. Carrillo, and J.~Soler.
\newblock Asymptotic behavior of an initial-boundary value problem for the
  {V}lasov-{P}oisson-{F}okker-{P}lanck system.
\newblock {\em SIAM J. Appl. Math.}, 57(5):1343--1372, 1997.

\bibitem[BD95]{BouchutDolbeault95}
F.~Bouchut and J.~Dolbeault.
\newblock On long time asymptotics of the {V}lasov-{F}okker-{P}lanck equation
  and of the {V}lasov-{P}oisson-{F}okker-{P}lanck system with {C}oulombic and
  {N}ewtonian potentials.
\newblock {\em Differential Integral Equations}, 8(3):487--514, 1995.

\bibitem[BGM10]{BolleyGuillinMalrieu2010}
F.~Bolley, A.~Guillin, and F.~Malrieu.
\newblock Trend to equilibrium and particle approximation for a weakly
  selfconsistent {V}lasov-{F}okker-{P}lanck equation.
\newblock {\em M2AN Math. Model. Numer. Anal.}, 44(5):867--884, 2010.

\bibitem[BKRS15]{BKRS2015}
V.~I. Bogachev, N.~V. Krylov, M.~R\"ockner, and S.~V. Shaposhnikov.
\newblock {\em Fokker-{P}lanck-{K}olmogorov equations}, volume 207 of {\em
  Mathematical Surveys and Monographs}.
\newblock American Mathematical Society, Providence, RI, 2015.

\bibitem[BT08]{BinneyTremaine2008}
J.~Binney and S.~Tremaine.
\newblock {\em {Galactic Dynamics: Second Edition}}.
\newblock Princeton University Press, 2008.

\bibitem[CJLW17]{Chazelle2017}
B.~Chazelle, Q.~Jiu, Q.~Li, and C.~Wang.
\newblock Well-posedness of the limiting equation of a noisy consensus model in
  opinion dynamics.
\newblock {\em J. Differential Equations}, 263(1):365--397, 2017.

\bibitem[CMV03]{carrillo2003}
J.~A. Carrillo, R.~J. McCann, and C.~Villani.
\newblock Kinetic equilibration rates for granular media and related equations:
  entropy dissipation and mass transportation estimates.
\newblock {\em Rev. Mat. Iberoamericana}, 19(3):971--1018, 12 2003.

\bibitem[CMV06]{Carrillo2006}
J.~A. Carrillo, R.~J. McCann, and C.~Villani.
\newblock Contractions in the 2-{W}asserstein length space and thermalization
  of granular media.
\newblock {\em Archive for Rational Mechanics and Analysis}, 179(2):217--263,
  Feb 2006.

\bibitem[CP10]{ChayesPanferov2010}
L.~Chayes and V.~Panferov.
\newblock The {M}c{K}ean-{V}lasov equation in finite volume.
\newblock {\em J. Stat. Phys.}, 138(1-3):351--380, 2010.

\bibitem[Daw83]{Dawson1983}
D.~A. Dawson.
\newblock Critical dynamics and fluctuations for a mean-field model of
  cooperative behavior.
\newblock {\em J. Statist. Phys.}, 31(1):29--85, 1983.

\bibitem[DE97]{DupuisEllis97}
Paul Dupuis and Richard~S Ellis.
\newblock {\em A {W}eak {C}onvergence {A}pproach to the {T}heory of {L}arge
  {D}eviations}, volume 902.
\newblock John Wiley \& Sons, 1997.

\bibitem[DG87]{DawsonGartner1987}
D.~A. Dawson and J.~G\"artner.
\newblock Large deviations from the {M}c{K}ean-{V}lasov limit for weakly
  interacting diffusions.
\newblock {\em Stochastics}, 20(4):247--308, 1987.

\bibitem[DLPS17]{Duong2017}
M.~H. Duong, A.~Lamacz, M.~A. Peletier, and U.~Sharma.
\newblock Variational approach to coarse-graining of generalized gradient
  flows.
\newblock {\em Calculus of Variations and Partial Differential Equations},
  56(4):100, Jun 2017.

\bibitem[DM76]{DymMcKean1976}
H.~Dym and H.~P. McKean.
\newblock {\em Gaussian processes, function theory, and the inverse spectral
  problem}.
\newblock Academic Press, New York, 1976.
\newblock Probability and Mathematical Statistics, Vol. 31.

\bibitem[DPZ13]{DPZ13}
M.~H. Duong, M.~A. Peletier, and J.~Zimmer.
\newblock G{ENERIC} formalism of a {V}lasov-{F}okker-{P}lanck equation and
  connection to large-deviation principles.
\newblock {\em Nonlinearity}, 26(11):2951--2971, 2013.

\bibitem[Dre87]{Dressler1987}
K.~Dressler.
\newblock Stationary solutions of the {V}lasov-{F}okker-{P}lanck equation.
\newblock {\em Math. Methods Appl. Sci.}, 9(2):169--176, 1987.

\bibitem[DT16]{DuongTugaut2016}
M.~H. Duong and J.~Tugaut.
\newblock Stationary solutions of the {V}lasov-{F}okker-{P}lanck equation:
  existence, characterization and phase-transition.
\newblock {\em Appl. Math. Lett.}, 52:38--45, 2016.

\bibitem[DT18]{DuongTugaut2018}
M.~H. Duong and J.~Tugaut.
\newblock The {V}lasov-{F}okker-{P}lanck equation in non-convex landscapes:
  convergence to equilibrium.
\newblock {\em Electron. Commun. Probab.}, 23:10 pp., 2018.

\bibitem[Duo15]{Duong2015NA}
M.~H. Duong.
\newblock Long time behaviour and particle approximation of a generalised
  {V}lasov dynamic.
\newblock {\em Nonlinear Anal.}, 127:1--16, 2015.

\bibitem[DV01]{DV2001}
L.~Desvillettes and C.~Villani.
\newblock On the trend to global equilibrium in spatially inhomogeneous
  entropy-dissipating systems: the linear {F}okker-{P}lanck equation.
\newblock {\em Comm. Pure Appl. Math.}, 54(1):1--42, 2001.

\bibitem[DZ78]{DesaiZwanzig1978}
R.~C. Desai and R.~Zwanzig.
\newblock Statistical mechanics of a nonlinear stochastic model.
\newblock {\em J. Statist. Phys.}, 19(1):1--24, 1978.

\bibitem[EPRB99]{EPR99}
J.-P. Eckmann, C.-A. Pillet, and L.~Rey-Bellet.
\newblock Non-equilibrium statistical mechanics of anharmonic chains coupled to
  two heat baths at different temperatures.
\newblock {\em Comm. Math. Phys.}, 201(3):657--697, 1999.

\bibitem[FKM65]{FKM65}
G.~W. Ford, M.~Kac, and P.~Mazur.
\newblock Statistical mechanics of assemblies of coupled oscillators.
\newblock {\em J. Mathematical Phys.}, 6:504--515, 1965.

\bibitem[Fle82]{Fleming1982}
W.~H. Fleming.
\newblock Logarithmic transformations and stochastic control.
\newblock In W.~H. Fleming and L.~G. Gorostiza, editors, {\em Advances in
  Filtering and Optimal Stochastic Control}, pages 131--141, Berlin,
  Heidelberg, 1982. Springer Berlin Heidelberg.

\bibitem[FNC{\etalchar{+}}16]{Fodor2016}
E.~Fodor, C.~Nardini, M.~E. Cates, J.~Tailleur, P.~Visco, and F.~van Wijland.
\newblock How far from equilibrium is active matter?
\newblock {\em Phys. Rev. Lett.}, 117:038103, Jul 2016.

\bibitem[Fra05]{Frank}
T.~D. Frank.
\newblock {\em Nonlinear {F}okker-{P}lanck equations}.
\newblock Springer Series in Synergetics. Springer-Verlag, Berlin, 2005.

\bibitem[FS17]{Farkhooi2017}
F.~Farkhooi and W.~Stannat.
\newblock Complete mean-field theory for dynamics of binary recurrent networks.
\newblock {\em Phys. Rev. Lett.}, 119:208301, Nov 2017.

\bibitem[G88]{gartner1988}
J.~G\"{a}rtner.
\newblock On the {M}c{K}ean-{V}lasov limit for interacting diffusions.
\newblock {\em Math. Nachr.}, 137:197--248, 1988.

\bibitem[GGS10]{giga2010}
M.H. Giga, Y.~Giga, and J.~Saal.
\newblock {\em Nonlinear Partial Differential Equations: Asymptotic Behavior of
  Solutions and Self-Similar Solutions}.
\newblock Progress in Nonlinear Differential Equations and Their Applications.
  Birkh{\"a}user Boston, 2010.

\bibitem[GKS04]{GKS04}
D.~Givon, R.~Kupferman, and A.~Stuart.
\newblock Extracting macroscopic dynamics: model problems and algorithms.
\newblock {\em Nonlinearity}, 17(6):R55--R127, 2004.

\bibitem[GNS{\etalchar{+}}12]{GodPavlal-2012}
B.~D. Goddard, A.~Nold, N.~Savva, G.~A. Pavliotis, and S.~Kalliadasis.
\newblock General dynamical density functional theory for classical fluids.
\newblock {\em Phys. Rev. Lett.}, 109:120603, Sep 2012.

\bibitem[GP18]{GomesPavliotis2018}
S.~N. Gomes and G.~A. Pavliotis.
\newblock Mean {F}ield {L}imits for {I}nteracting {D}iffusions in a
  {T}wo-{S}cale {P}otential.
\newblock {\em J. Nonlinear Sci.}, 28(3):905--941, 2018.

\bibitem[GPK12]{GodPavlKall11}
B.~D. Goddard, G.~A. Pavliotis, and S.~Kalliadasis.
\newblock The overdamped limit of dynamic density functional theory: rigorous
  results.
\newblock {\em Multiscale Model. Simul.}, 10(2):633--663, 2012.

\bibitem[GPY13]{GPY2013}
J.~Garnier, G.~Papanicolaou, and T.-W. Yang.
\newblock Large deviations for a mean field model of systemic risk.
\newblock {\em SIAM J. Financial Math.}, 4(1):151--184, 2013.

\bibitem[GPY17]{GPY2017}
J.~Garnier, G.~Papanicolaou, and T.-W. Yang.
\newblock Consensus convergence with stochastic effects.
\newblock {\em Vietnam J. Math.}, 45(1-2):51--75, 2017.

\bibitem[HJ95]{HanggiJung1995}
P.~Hanggi and P.~Jung.
\newblock {Colored noise in dynamical systems}.
\newblock In {Prigogine, I and Rice, SA}, editor, {\em {Advances in Chemical
  Physics, VOL LXXXIX}}, volume~{89} of {\em {Advances in Chemical Physics}},
  pages {239--326}. {1995}.

\bibitem[HL84]{HorsLef84}
W.~Horsthemke and R.~Lefever.
\newblock {\em Noise-induced transitions}, volume~15 of {\em Springer Series in
  Synergetics}.
\newblock Springer-Verlag, Berlin, 1984.
\newblock Theory and applications in physics, chemistry, and biology.

\bibitem[Kol34]{Kol34}
A.~Kolmogoroff.
\newblock Zufallige bewegungen (zur theorie der brownschen bewegung).
\newblock {\em Annals of Mathematics}, 35(1):116--117, 1934.

\bibitem[Kup04]{Kup03}
R.~Kupferman.
\newblock Fractional kinetics in {K}ac-{Z}wanzig heat bath models.
\newblock {\em J. Statist. Phys.}, 114(1-2):291--326, 2004.

\bibitem[LB07]{LorenziBertoldi2007}
L.~Lorenzi and M.~Bertoldi.
\newblock {\em Analytical methods for Markov semigroups}.
\newblock Chapman \& Hall/CRC, 2007.

\bibitem[LP85]{LindquistPicci1985}
A.~Lindquist and G.~Picci.
\newblock Realization theory for multivariate stationary {G}aussian processes.
\newblock {\em SIAM J. Control Optim.}, 23(6):809--857, 1985.

\bibitem[LS16]{Lucon2016}
E.~Lu\c{c}on and W.~Stannat.
\newblock Transition from {G}aussian to non-{G}aussian fluctuations for
  mean-field diffusions in spatial interaction.
\newblock {\em Ann. Appl. Probab.}, 26(6):3840--3909, 2016.

\bibitem[McK66]{McKean1966}
H.~P. McKean, Jr.
\newblock A class of {M}arkov processes associated with nonlinear parabolic
  equations.
\newblock {\em Proc. Nat. Acad. Sci. U.S.A.}, 56:1907--1911, 1966.

\bibitem[McK67]{McKean1967}
H.~P. McKean, Jr.
\newblock Propagation of chaos for a class of non-linear parabolic equations.
\newblock In {\em Stochastic {D}ifferential {E}quations ({L}ecture {S}eries in
  {D}ifferential {E}quations, {S}ession 7, {C}atholic {U}niv., 1967)}, pages
  41--57. Air Force Office Sci. Res., Arlington, Va., 1967.

\bibitem[Mie11]{Mielke11}
A.~Mielke.
\newblock Formulation of thermoelastic dissipative material behavior using
  {GENERIC}.
\newblock {\em Contin. Mech. Thermodyn.}, 23(3):233--256, 2011.

\bibitem[MKD17]{Mandal2017}
D.~Mandal, K.~Klymko, and M.~R. DeWeese.
\newblock Entropy production and fluctuation theorems for active matter.
\newblock {\em Phys. Rev. Lett.}, 119:258001, Dec 2017.

\bibitem[Mon17]{Monmarche2017}
P.~Monmarch\'e.
\newblock Long-time behaviour and propagation of chaos for mean field kinetic
  particles.
\newblock {\em Stochastic Process. Appl.}, 127(6):1721--1737, 2017.

\bibitem[MPP02]{Metafune2002}
G.~Metafune, D.~Pallara, and E.~Priola.
\newblock Spectrum of ornstein-uhlenbeck operators in lp spaces with respect to
  invariant measures.
\newblock {\em Journal of Functional Analysis}, 196(1):40 -- 60, 2002.

\bibitem[MT14]{Motsch2014}
S.~Motsch and E.~Tadmor.
\newblock Heterophilious dynamics enhances consensus.
\newblock {\em SIAM Rev.}, 56(4):577--621, 2014.

\bibitem[Ngu18]{Nguyen2018}
H.~Nguyen.
\newblock The small-mass limit and white-noise limit of an infinite dimensional
  generalized langevin equation, 2018.

\bibitem[Oel84]{oelschlager1984}
K.~Oelschl\"ager.
\newblock A martingale approach to the law of large numbers for weakly
  interacting stochastic processes.
\newblock {\em Ann. Probab.}, 12(2):458--479, 1984.

\bibitem[OG97]{OG97a}
H.~C. \"Ottinger and M.~Grmela.
\newblock Dynamics and thermodynamics of complex fluids. {II}. {I}llustrations
  of a general formalism.
\newblock {\em Phys. Rev. E (3)}, 56(6):6633--6655, 1997.

\bibitem[OP11]{OttobrePavliotis11}
M.~Ottobre and G.~A. Pavliotis.
\newblock Asymptotic analysis for the generalized {L}angevin equation.
\newblock {\em Nonlinearity}, 24(5):1629--1653, 2011.

\bibitem[OPPS12]{OPS2012}
M.~Ottobre, G.A. Pavliotis, and K.~Pravda-Starov.
\newblock Exponential return to equilibrium for hypoelliptic quadratic systems.
\newblock {\em Journal of Functional Analysis}, 262(9):4000 -- 4039, 2012.

\bibitem[OPPS15]{OPS2015}
M.~Ottobre, G.A. Pavliotis, and K.~Pravda-Starov.
\newblock Some remarks on degenerate hypoelliptic ornstein–uhlenbeck
  operators.
\newblock {\em Journal of Mathematical Analysis and Applications}, 429(2):676
  -- 712, 2015.

\bibitem[{\"{O}}tt05]{Oettinger05}
H.~C. {\"{O}}ttinger.
\newblock {\em Beyond Equilibrium Thermodynamics}.
\newblock Wiley-Interscience, 2005.

\bibitem[Ott11]{OttobrePhD}
M.~Ottobre.
\newblock {\em Asymptotic Analysis for Markovian models in non-equilibrium
  Statistical Mechanics}.
\newblock PhD thesis, Imperial College London, 2011.

\bibitem[Pav14]{Pavl2014}
G.~A. Pavliotis.
\newblock {\em Stochastic processes and applications}, volume~60 of {\em Texts
  in Applied Mathematics}.
\newblock Springer, New York, 2014.
\newblock Diffusion processes, the Fokker-Planck and Langevin equations.

\bibitem[PS08]{PavliotisStuart2008}
G.A. Pavliotis and A.M. Stuart.
\newblock {\em Multiscale methods}, volume~53 of {\em Texts in Applied
  Mathematics}.
\newblock Springer, New York, 2008.
\newblock Averaging and homogenization.

\bibitem[RB06]{Rey-Bellet2006}
L.~Rey-Bellet.
\newblock Open classical systems.
\newblock In {\em Open quantum systems. {II}}, volume 1881 of {\em Lecture
  Notes in Math.}, pages 41--78. Springer, Berlin, 2006.

\bibitem[Shi87]{Shiino1987}
M.~Shiino.
\newblock Dynamical behavior of stochastic systems of infinitely many coupled
  nonlinear oscillators exhibiting phase transitions of mean-field type: H
  theorem on asymptotic approach to equilibrium and critical slowing down of
  order-parameter fluctuations.
\newblock {\em Phys. Rev. A}, 36:2393--2412, Sep 1987.

\bibitem[Sno07]{snook07}
I.~Snook.
\newblock {\em The {L}angevin and generalized {L}angevin approach to the
  dynamics of atomic, polymeric and colloidal systems}.
\newblock Elsevier, Amsterdam, 2007.

\bibitem[Tam84]{Tamura1984}
Y.~Tamura.
\newblock On asymptotic behaviors of the solution of a nonlinear diffusion
  equation.
\newblock {\em J. Fac. Sci. Univ. Tokyo Sect. IA Math.}, 31(1):195--221, 1984.

\bibitem[Tam87]{Tamura1987}
Y.~Tamura.
\newblock Free energy and the convergence of distributions of diffusion
  processes of {M}c{K}ean type.
\newblock {\em J. Fac. Sci. Univ. Tokyo Sect. IA Math.}, 34(2):443--484, 1987.

\bibitem[Tro77]{tropper1977}
M.~M. Tropper.
\newblock Ergodic and quasideterministic properties of finite-dimensional
  stochastic systems.
\newblock {\em J. Statist. Phys.}, 17(6):491--509, 1977.

\bibitem[Tug13a]{Tugaut13a}
J.~Tugaut.
\newblock Convergence to the equilibria for self-stabilizing processes in
  double-well landscape.
\newblock {\em Ann. Probab.}, 41(3A):1427--1460, 2013.

\bibitem[Tug13b]{Tugaut13b}
J.~Tugaut.
\newblock Self-stabilizing processes in multi-wells landscape in
  {$\Bbb{R}^d$}-convergence.
\newblock {\em Stochastic Process. Appl.}, 123(5):1780--1801, 2013.

\bibitem[Tug14]{Tugaut2014}
J.~Tugaut.
\newblock Phase transitions of {M}c{K}ean-{V}lasov processes in double-wells
  landscape.
\newblock {\em Stochastics}, 86(2):257--284, 2014.

\bibitem[Vil09]{Villani2009}
C.~Villani.
\newblock Hypocoercivity.
\newblock {\em Mem. Amer. Math. Soc.}, 202(950):iv+141, 2009.

\bibitem[WLEC17]{Chazelle_al2017}
C.~Wang, Q.~Li, W.~E, and B.~Chazelle.
\newblock Noisy {H}egselmann-{K}rause systems: phase transition and the
  {$2R$}-conjecture.
\newblock {\em J. Stat. Phys.}, 166(5):1209--1225, 2017.

\end{thebibliography}
\end{document}